\documentclass[12pt,a4paper]{amsart}
\usepackage{amsfonts,color}
\usepackage{amsthm}
\usepackage{amsmath}
\usepackage{amscd}
\usepackage[latin2]{inputenc}
\usepackage{t1enc}
\usepackage[mathscr]{eucal}
\usepackage{indentfirst}
\usepackage{graphicx}
\usepackage{graphics}
\usepackage{pict2e}
\usepackage{epic}
\usepackage{url}
\usepackage{epstopdf}
\usepackage{comment}
\usepackage{todonotes}
\usepackage{stmaryrd}

\allowdisplaybreaks

\numberwithin{equation}{section}
\usepackage[margin=2.2cm]{geometry}

\usepackage{pgfplots}
\usepackage{xcolor}
\usepackage{tikz}
\usetikzlibrary{matrix,arrows,decorations.pathmorphing}
\usetikzlibrary{calc,decorations.pathreplacing}
\usetikzlibrary{quotes,angles}
\usetikzlibrary{shapes}
\usetikzlibrary{patterns}

\tikzstyle{vertex}=[draw=black,circle,fill=black,minimum size=6pt, inner sep=0pt, outer sep=0pt,text=black,line width=0mm]

\tikzstyle{Sqvertex}=[draw=black,shape=rectangle, minimum size=10pt, fill=white]

\tikzstyle{Cvertex}=[draw=black,shape=circle, minimum size=6pt, fill=white]

\tikzstyle{vertex_blue}=[draw=black,circle,fill=blue,minimum size=6pt, inner sep=0pt, outer sep=0pt,text=black,line width=0mm]
\tikzstyle{vertex_red}=[draw=black,circle,fill=red,minimum size=6pt, inner sep=0pt, outer sep=0pt,text=black,line width=0mm]
\tikzstyle{vertex_green}=[draw=black,circle,fill=green,minimum size=6pt, inner sep=0pt, outer sep=0pt,text=black,line width=0mm]
\tikzstyle{c0}=[shape=circle, minimum size=4pt, fill=white]
\tikzstyle{c1}=[shape=rectangle, minimum size=7pt, fill=red]
\tikzstyle{c2}=[shape=diamond, minimum size=10pt, fill=blue]
\tikzstyle{mybox} = [rectangle, rounded corners, minimum width=3cm, minimum height=1cm,text centered, draw=black]
\tikzset{base/.style = {rectangle, rounded corners, draw=black,
                           minimum width=3cm, minimum height=1cm,
                           text centered}}

\usepgfplotslibrary{fillbetween}
\pgfplotsset{mystyle/.style={%
        xmin=-2,
        xmax=7.9,
        ymin=-1,
        xtick = {1,3},
        xticklabels = {{1},$d-1$},
        ytick = {1}
    }
}

\pgfplotsset{mystyle2/.style={%
        xmin=-2,
        xmax=7.9,
        ymin=-1,
        ymax=6,
        xtick = {1,3},
        xticklabels = {{1},$d-1$},
        ytick = {1}
    }
}

\definecolor{darkerblue}{HTML}{065A82} 
\definecolor{lighterblue}{HTML}{1C7293} 
\theoremstyle{plain}
\newtheorem{Th}{Theorem}[section]
\newtheorem{Lemma}[Th]{Lemma}

\newtheorem{Prop}[Th]{Proposition}

 \theoremstyle{definition}
\newtheorem{Def}[Th]{Definition}

\newtheorem{Rem}[Th]{Remark}
\newtheorem{?}[Th]{Problem}
\newtheorem{Ex}[Th]{Example}

\newcommand{\veca}{\underline{a}}
\newcommand{\vecb}{\underline{b}}
\newcommand{\vecv}{\underline{v}}
\newcommand{\vecmu}{\underline{\mu}}
\newcommand{\x}{\underline{x}}
\newcommand{\y}{\underline{y}}
\newcommand{\ovt}{\overline{t}}
\newcommand{\ova}{\overline{a}}
\newcommand{\ovb}{\overline{b}}

\newcommand{\phiab}{\Phi_{\veca,\vecb,\vecmu}}

\newcommand{\F}{\overline{\mathbb{F}}}

\begin{document}

\title{Random cluster model on regular graphs}

\author[F. Bencs]{Ferenc Bencs}

\address{Korteweg de Vries Institute for Mathematics, University of Amsterdam. P.O. Box 94248 1090 GE
Amsterdam The Netherlands}
\email{ferenc.bencs@gmail.com}

\author[M. Borb\'enyi]{M\'arton Borb\'enyi}

\address{E\"otv\"os Lor\'and University, H-1117 Budapest, P\'azm\'any P\'eter s\'et\'any 1/C \and Alfr\'ed R\'enyi Institute of Mathematics, H-1053 Budapest, Re\'altanoda utca 13-15}
\email{marton.borbenyi@gmail.com}

\author[P. Csikv\'ari]{P\'{e}ter Csikv\'{a}ri}

\address{Alfr\'ed R\'enyi Institute of Mathematics, H-1053 Budapest, Re\'altanoda utca 13-15 \and E\"otv\"os Lor\'and University, H-1117 Budapest, P\'azm\'any P\'eter s\'et\'any 1/C }
\email{peter.csikvari@gmail.com}

\thanks{The first author is supported by the NKFIH (National Research, Development and Innovation Office, Hungary) grant KKP-133921. The second author is supported by the UNKP-21-2 New National Excellence Program of the Ministry for Innovation and Technology from the source of the National Research, Development and Innovation Fund.
The third author  is  supported by the  Counting in Sparse Graphs Lend\"ulet Research Group.
}

 \subjclass[2010]{Primary: 05C30. Secondary: 05C31, 05C70}

 \keywords{random cluster model, Tutte polynomial, Bethe approximation} 

\begin{abstract}
For a graph $G=(V,E)$ with $v(G)$ vertices the partition function of the random cluster model is defined by
$$Z_G(q,w)=\sum_{A\subseteq E(G)}q^{k(A)}w^{|A|},$$
where $k(A)$ denotes the number of connected components of the graph $(V,A)$. Furthermore, let $g(G)$ denote the girth of the graph $G$, that is,  the length of the shortest cycle.

In this paper we show that if $(G_n)_n$ is a sequence of $d$-regular graphs such that the girth $g(G_n)\to \infty$, then
the limit 
$$\lim_{n\to \infty} \frac{1}{v(G_n)}\ln Z_{G_n}(q,w)=\ln \Phi_{d,q,w}$$
exists if $q\geq 2$ and $w\geq 0$. The quantity $\Phi_{d,q,w}$ can be computed as follows. Let
$$\Phi_{d,q,w}(t):=\left(\sqrt{1+\frac{w}{q}}\cos(t)+\sqrt{\frac{(q-1)w}{q}}\sin(t)\right)^{d}+(q-1)\left(\sqrt{1+\frac{w}{q}}\cos(t)-\sqrt{\frac{w}{q(q-1)}}\sin(t)\right)^{d},$$
then
$$\Phi_{d,q,w}:=\max_{t\in [-\pi,\pi]}\Phi_{d,q,w}(t),$$
The same conclusion holds true for a sequence of random $d$-regular graphs with probability one.

 Our result  extends the work of Dembo, Montanari, Sly and Sun for the Potts model (integer $q$), and we prove a conjecture of Helmuth, Jenssen and Perkins about the phase transition of the random cluster model with fixed $q$.
\end{abstract}

\maketitle

\section{Introduction}
For a graph $G=(V,E)$ the partition function of the random cluster model is defined by
$$Z_G(q,w)=\sum_{A\subseteq E(G)}q^{k(A)}w^{|A|},$$
where $k(A)$ denotes the number of connected components of the graph $(V,A)$. In many papers, one uses the parametrization $w=e^{\beta}-1$.

When $q$ is a positive integer, then $Z_G(q,w)$ is also the partition function of the Potts-model with $q$ spins, moreover there is a natural coupling between the two models, see for example \cite{grimmett2004random}. In this paper we call a model a spin model with $r$ spins if there is an $r\times r$ symmetric matrix $N$   and a vector $\vecmu \in \mathbb{R}^r$ such that for a graph $G=(V,E)$ the probability of a $\sigma: V\to \{1,2,\dots ,r\}$ is
$$\mathbb{P}(\sigma)=\frac{1}{Z_G(N,\vecmu)}\prod_{v\in V}\mu_{\sigma(v)}\prod_{(u,v)\in E(G)}N_{\sigma(u),\sigma(v)},$$
where with the notation $[r]=\{1,2,\dots ,r\}$ we have
$$Z_G(N,\vecmu)=\sum_{\sigma:V\to [r]}\prod_{v\in V}\mu_{\sigma(v)} \prod_{(u,v)\in E(G)}N_{\sigma(u),\sigma(v)}.$$
In both expressions the second product is over the edge set $E(G)$, the symmetricity of $N$ ensures that the expression is well-defined.  The quantity $Z_G(N,\vecmu)$ is the partition function of the model. In case of the Potts-model we have $r=q$ and $N=J_q+wI_q$, where $J_q$ is the $q\times q$ matrix consisting of $1$'s and $I_q$ is the $q\times q$ identity matrix. The vector $\vecmu$ is the constant $1$ vector in this case. In general, if $\vecmu$ is the constant $1$ vector, then we will simply write $Z_G(N)$ instead of $Z_G(N,\vecmu)$.
\bigskip

Let $v(G)$ denote the number of vertices of a graph $G$. In this paper we study the quantity
$$\lim_{n\to \infty} \frac{1}{v(G_n)}\ln Z_{G_n}(q,w)$$
when $(G_n)_n$ is an essentially large girth sequence of  $d$-regular graphs. A graph sequence $(G_n)_n$ is called essentially large girth if for all $g$ we have $\lim_{n\to \infty}\frac{L(G_n,g)}{v(G_n)}=0$, where $L(G,g)$ denotes the number of cycles of length at most $g-1$. It is known that a sequence of random $d$-regular graphs is essentially large girth graph sequence with probability one (see for instance \cite{mckay2004short}). So the problems of determining
$\lim_{n\to \infty} \frac{1}{v(G_n)}\mathbb{E}\ln Z_{G_n}(q,w)$ or $\lim_{n\to \infty} \frac{1}{v(G_n)}\ln \mathbb{E}Z_{G_n}(q,w)$ for random $d$-regular graph sequence $(G_n)_n$ are very strongly related to this question. In fact, it will turn out that all these limits are the same. The main theorem of this paper is the following.

\begin{Th} \label{main}  If $(G_n)_n$ is an essentially large girth sequence of  $d$-regular graphs, then
the limit 
$$\lim_{n\to \infty} \frac{1}{v(G_n)}\ln Z_{G_n}(q,w)=\ln \Phi_{d,q,w}$$
exists for $q\geq 2$ and $w\geq 0$. The quantity $\Phi_{d,q,w}$  can be computed as follows. Let 
$$\Phi_{d,q,w}(t):=\left(\sqrt{1+\frac{w}{q}}\cos(t)+\sqrt{\frac{(q-1)w}{q}}\sin(t)\right)^{d}+(q-1)\left(\sqrt{1+\frac{w}{q}}\cos(t)-\sqrt{\frac{w}{q(q-1)}}\sin(t)\right)^{d},$$
then
$$\Phi_{d,q,w}:=\max_{t\in [-\pi,\pi]}\Phi_{d,q,w}(t).$$
The same conclusion holds true with probability one for a sequence of random $d$-regular graphs. 
\end{Th}

The quantity $\Phi_{d,q,w}$ has various alternative descriptions. As far as we know the description in the theorem is new even for the Potts model. There is a critical value $w_c(d,q)$ such that if $0\leq w\leq w_c(d,q)$, then  $\Phi_{d,q,w}(t)$ is maximized at $t=0$, and so $\Phi_{d,q,w}=q\left(1+\frac{w}{q}\right)^{d/2}$, and for $w>w_c(d,q)$ we have $\Phi_{d,q,w}>q\left(1+\frac{w}{q}\right)^{d/2}$. Moreover, if $q>2$, then $\frac{\partial}{\partial w }\Phi_{d,q,w}$ is discontinuous at $w_c(q)$, that is, there is a first order phase transition at $w_c(q)$.
We will see that
$$w_c(d,q)=\frac{q-2}{(q-1)^{1-2/d}-1}-1.$$
For random $d$-regular graphs and the Potts model this was established by  Galanis, \v{S}tefankovi\v{c},
Vigoda, and Yang \cite{galanis2016ferromagnetic}. For not necessarily integer $q\geq 2$ this was conjectured by Helmuth, Jenssen and Perkins \cite{helmuth2020finite}. We will also prove that for any $d$-regular graph $G$ we have
$$Z_{G}(q,w)\geq \Phi_{d,q,w}^{v(G)},$$
and we also show that if $G$ contains $\varepsilon v(G)$ cycles of length at most $g$ for some fixed $\varepsilon$ and $g$, then $Z_{G}(q,w)$ is exponentially larger than that bound.  Related results were obtained by Ruozzi \cite{ruozzi2013beyond}.

\subsection{Related works.} There has been a lot of work on Potts model and random cluster model both on random regular graphs and essentially large girth graph sequences. The papers mentioned below treat various related problems, but we only mention the results that are directly related to Theorem~\ref{main}.

The case $q=2$ and $w\geq 0$ is the so-called ferromagnetic Ising-model. In this case, Dembo and Montanari~\cite{dembo2010ising} proved that $Z_{G_n}(2,w)^{1/v(G_n)}$ converges if $(G_n)_n$ is an  essentially large girth sequence of $d$-regular graphs. In fact, they proved a significantly more general theorem about  essentially large girth  graphs that are  not necessarily regular. Note that they use the terminology locally tree-like for what we use essentially large girth. When $q$ is a positive integer and $w\geq 0$, that is, in the case of the ferromagnetic Potts model, Dembo, Montanari and Sun~\cite{dembo2013factor} proved the convergence of  $Z_{G_n}(q,w)^{1/v(G_n)}$ for essentially large girth sequence of $d$-regular graphs for every $d$  except when $w$  belongs to a certain interval $(w_0,w_1)$. Later Dembo, Montanari, Sly and Sun~\cite{dembo2014replica} proved the convergence of $Z_{G_n}(q,w)^{1/v(G_n)}$ for   essentially large girth sequence of $d$-regular graphs, when $d$ is even, $q$ is a positive integer and $w\geq 0$ even if $w\in (w_0,w_1)$. Very recently, Helmuth, Jenssen and Perkins \cite{helmuth2020finite} proved the convergence of  $Z_{G_n}(q,w)^{1/v(G_n)}$ for essentially  large girth sequence of $d$-regular graphs for large (not necessarily integer) $q$ and $w\geq 0$ with the additional hypothesis that $(G_n)_n$ satisfies some expansion condition. This line of research using cluster expansion and an expansion property of $(G_n)_n$ was extended by Carlson, Davies and Kolla \cite{carlson2020efficient} and by Carlson, Davies, Fraiman,  Kolla, Potukuchi, and Yap \cite{carlson2022algorithms} for the Potts model. Ferromagnetic Potts models on random regular graphs were also studied by Galanis, \v{S}tefankovi\v{c}, Vigoda and Yang \cite{galanis2016ferromagnetic}.

Our theorem does not cover the case when $w<0$. When $q$ is a positive integer and $w=-1$, then $Z_G(q,-1)$ counts the number of proper colorings of the graph $G$. This case was treated by Bandyopadhyay and Gamarnik \cite{bandyopadhyay2008counting}.  They showed that if $q\geq d+1$, then for an essentially large girth graph sequence of $d$-regular graphs $(G_n)_n$ we have
$$\lim_{n\to \infty}Z_{G_n}(q,-1)^{1/v(G_n)}=q\left(1-\frac{1}{q}\right)^{d/2}.$$
Their result was extended for integer $q\geq 2\Delta$ and $w\geq -1$ by Borgs, Chayes, Kahn and Lov\'asz \cite{borgs2013left} for general Benjamini--Schramm convergent graph sequences, where $\Delta$ is a bound on the degrees of all $G_n$ in the sequence. The result of Bandyopadhyay and Gamarnik \cite{bandyopadhyay2008counting} was extended for not necessarily integer $q\geq 8\Delta$ and $w=-1$ by Ab\'ert and Hubai \cite{abert2015benjamini} also for arbitrary  Benjamini--Schramm convergent graph sequences. Csikv\'ari and Frenkel \cite{csikvari2016benjamini} showed that the same conclusion holds true for every fixed $w\geq 0$ and $q$ sufficiently large in terms of $w$ and $\Delta$. The partition function of the random cluster model is strongly related to the so-called Tutte polynomial \cite{tutte1954contribution}.
For a graph $G=(V,E)$ the Tutte polynomial $T_G(x,y)$ is defined by  
$$T_G(x,y)=\sum_{A\subseteq E}(x-1)^{k(A)-k(E)}(y-1)^{k(A)+|A|-v(G)}.$$ The connection between the Tutte polynomial and the random cluster model is
$$T_G(x,y)=(x-1)^{-k(E)}(y-1)^{-v(G)}Z_G((x-1)(y-1),y-1).$$
Bencs and Csikv\'ari \cite{bencs2021evaluations} proved that for an essentially large girth sequence of $d$-regular graphs $(G_n)_n$ we have
$$\lim_{n\to \infty}T_{G_n}(x,y)^{1/v(G_n)}=\left\{\begin{array}{lc} (d-1)\left(\frac{(d-1)^2}{(d-1)^2-x}\right)^{d/2-1}&\ \ \mbox{if}\ x\leq d-1,\\
x\left(1+\frac{1}{x-1}\right)^{d/2-1} &\ \ \mbox{if}\ x> d-1
\end{array}\right.$$
for $x\geq 1$ and $0\leq y\leq 1$. The next theorem  summarizes the known results for the Tutte polynomial for non-negative $x,y$.

\begin{Th} \label{th: Tutte-regions}
Let $(G_n)_n$ be an essentially large girth sequence of $d$-regular graphs. Then the limit 
$$\lim_{n\to \infty}T_{G_n}(x,y)^{1/v(G_n)}=t_d(x,y)$$
exists if $x,y$ satisfy one the following conditions\\
(i) (Theorem~\ref{main}) $(x-1)(y-1)\geq 2$ and $y>1$,\\
(ii) (see Section~\ref{sec:related}) $x\geq d-1$ and $y\geq 0$,\\
(iii) (Bencs and Csikv\'ari \cite{bencs2021evaluations}) $x\geq 1$ and $0\leq y\leq 1$.
\end{Th}

Figure~\ref{fig1} depicts the regions described in Theorem~\ref{th: Tutte-regions}.

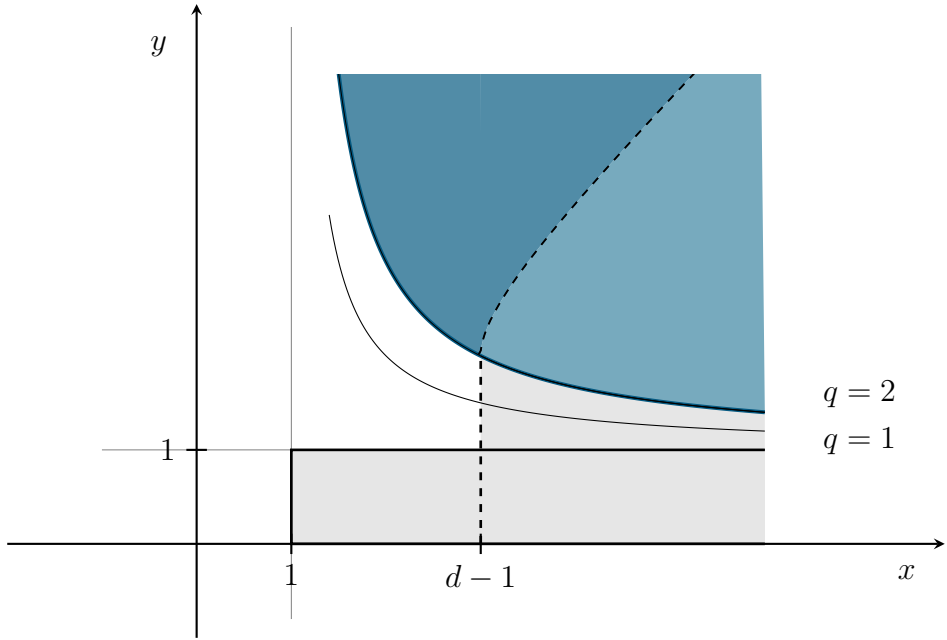
\begin{figure}[h!] 
\centering
\begin{tikzpicture}[scale=1.8]

    \begin{axis}[
            axis lines=middle,
            samples=200,
            mystyle,
            axis equal image,
            axis line style = thick,
            every tick/.style={
                thick,
            },
            axis on top = true
        ]
        \addplot[black,domain=1.4:6] {1/(x-1) + 1};
        \addplot[name path=f,  darkerblue, domain=1.435:6, line width=1.5 pt] {2/(x-1) + 1};
        \addplot[name path=fbal, line width=0.1 pt, domain=1.435:3] {2/(x-1) + 1};
        \addplot[name path=fjobb, line width=0.1 pt, domain=3:6] {2/(x-1) + 1};
        \path[name path=g] (axis cs:1.5,2/0.5+1) -- (axis cs:6,2/0.5+1);
        \path[name path=gbal] (axis cs:1.5,2/0.5+1.6) -- (axis cs:3,2/0.5+1.6);
        \path[name path=gjobb] (axis cs:3,2/0.5+1.6) -- (axis cs:6,2/0.5+1.6);
    
        \addplot [domain=2.01:22,dashed, line width = 0.7 pt] ({x/((x-2)/((x-1)^(1/2)-1)-1)+1},{(x-2)/((x-1)^(1/2)-1)});
        \addplot [domain=2.01:23.5,dashed, line width = 0 pt, name path = fazishatar] ({x/((x-2)/((x-1)^(1/2)-1)-1)+1},{(x-2)/((x-1)^(1/2)-1)});
        %\addplot[black,domain=1.8:6] {5/(x-1) + 1};
        \node (A) at (axis cs:7,1.1) {$q=1$};
        \node (A) at (axis cs:7,1.6) {$q=2$};
        \node (A) at (axis cs:7.5,-0.3) {$x$};
        \node (A) at (axis cs:-0.4,5.3) {$y$};
%         \path[name path=fuggobal] (axis cs:3,0) -- (axis cs:3,2);
%         \path[name path=fuggojobb] (axis cs:6,0) -- (axis cs:6,2/0.55+1);
        
        % asszimptotak;
        \draw[gray, line width=0.2 pt] (axis cs:1,-0.8) -- (axis cs:1,5.5);
        \draw[gray, line width=0.2 pt] (axis cs:-1,1) -- (axis cs:6,1);

        \addplot[lighterblue!60] fill between[of=fjobb and fazishatar]; % rank2
        \addplot[darkerblue!70] fill between[of=fazishatar and gjobb]; % rank2
        \addplot[darkerblue!70] fill between[of=fbal and gbal]; % rank2
        \draw[draw=none,fill=gray!20] (axis cs:1,0) rectangle (axis cs:6,1); % forest
        \addplot [domain=3:6,name path = also] {1};
        \addplot [domain=3:6,name path = felso] {2/(x-1) + 1};
        \addplot[gray!20, domain=3:6] fill between[of=also and felso]; % harmadik szin
        
        \draw[black,line width=1pt] (axis cs:6,0) -- (axis cs:1,0) -- (axis cs:1,1) -- (axis cs:6,1);

        \draw[dashed,line width=1pt] (axis cs:3,0) -- (axis cs:3,2);
        \draw[draw=none,fill=white] (axis cs:1.2,5) rectangle (axis cs:6,6); % forest

    \end{axis}

\end{tikzpicture} 
\caption{The investigated parameters of the article (in blue). For description of the related regions (in gray) see Section~\ref{sec:related}. The dashed lines are $x=d-1$ and the phase transition parametrized in $x,y$. }
\label{fig1}
\end{figure}

\subsection{Plan of the paper.} This paper has essentially two parts. In the first part we show that $Z_G(q,w)$ can be approximated by the partition function of a $2$-spin model for essentially large girth graphs, namely $Z_G(q,w)\approx Z_G(M'_2,\underline{\nu}_2)$, where
$$M'_2=\left(\begin{array}{cc}
            1+w & 1   \\
            1   & 1+\frac{w}{q-1}
            \end{array}\right)\ \ \ \text{and}\ \ \ \underline{\nu}_2=\left(\begin{array}{c}
            1   \\
            q-1
            \end{array}\right).$$
The precise statement is Theorem~\ref{main approximation}. In this theorem we do not use that $G$ is regular so we believe that this statement is very useful for studying random cluster model on other essentially large girth graphs like  Erd\H os-R\'enyi random graphs $G(n,\frac{c}{n})$. This statement implies that for an essentially large girth sequence of $d$-regular graphs $(G_n)_n$ we have
$$\lim_{n\to \infty} \frac{1}{v(G_n)}\ln Z_{G_n}(q,w)=\lim_{n\to \infty} \frac{1}{v(G_n)}\ln Z_{G_n}(M'_2,\underline{\nu}_2).$$

At that point one can simply cite a theorem of Sly and Sun \cite{sly2012computational,sly2014counting} (relying on a theorem of Dembo and Montanari \cite{dembo2010ising}) that shows that the aforementioned limit exists since $\det(M'_2)>0$. Indeed, Sly and Sun \cite{sly2012computational,sly2014counting} proved that for regular graphs any $2$-spin model $(N,\vecmu)$ having a  positive determinant  is equivalent with a ferromagnetic Ising model. Dembo and Montanari \cite{dembo2010ising} showed that for the ferromagnetic Ising-model the limit indeed exists. In these papers the main technique is an abstract interpolation method. Nevertheless, in this paper we do not rely on these papers, instead we build out a little theory for ``ferromagnetic'' $2$-spin models that builds on Lee-Yang theory \cite{lee1952statistical,wagner2009weighted,yang1952statistical} and the gauge theory of Chertkov and Chernyak \cite{chertkov2006loop2,chertkov2006loop1}. This approach has some additional gains. First of all, it shows that the limit exists not only for essentially large girth sequence of $d$-regular graphs but for Benjamini--Schramm convergent sequence of $d$-regular graphs (for the definition of a  Benjamini--Schramm convergent graph sequence see Definition~\ref{def: BS-convergence}, for the precise statement see Theorem~\ref{convergence_theorem}). Secondly, we  prove a stability theorem that shows that if a $d$-regular graph contains a linear number of short cycles, that is, it contains at least $\varepsilon v(G)$ cycles of length at most $g$ for some fixed $\varepsilon$ and $g$, then both $Z_G(q,w)$ and $Z_G(N,\vecmu)$ are exponentially larger than the number obtained for an essentially large girth sequence of $d$-regular graphs (for the precise statement see Theorem~\ref{lower-bound} and the remark after the theorem).

So the second part of the paper is an elaborate analysis of the  $2$-spin model $(N,\vecmu)$, where $N$ is a positive definite $2\times 2$ matrix with positive entries and $\vecmu$ is a vector in $\mathbb{R}^2$ with positive entries. We will show that there exists a quantity $\Phi_d(N,\vecmu)$ such that for every essentially large girth sequence of $d$-regular graphs $(G_n)_n$ we have
$$\lim_{n\to \infty}\frac{1}{v(G_n)}\ln Z_{G_n}(N,\vecmu)=\ln \Phi_d(N,\vecmu).$$
To analyse these models we use a strategy that can be of independent interest. The idea is that we can associate many different polynomials to the same computational problem and the zeros of one of these polynomials satisfy Lee-Yang theorem, that is, they are on a circle. Indeed, for a $d$-regular graph $G$ let
 us introduce the following polynomials \cite{borbenyi2020counting,wagner2009weighted}
$$F_G(x_0,\dots ,x_d)=\sum_{A\subseteq E}\left(\prod_{v\in V}
     x_{d_A(v)}\right),$$
and a bit more generally,
$$F_G(x_0,\dots ,x_d|z)=\sum_{A\subseteq E}\left(\prod_{v\in V}
     x_{d_A(v)}\right)z^{2|A|}=F_G(x_0,x_1z,x_2z,...,x_dz^d).$$
We call $F_G(x_0,\dots ,x_d)$ and $F_G(x_0,\dots ,x_d|z)$ the subgraph counting polynomial. 
     
We show that if $N$ is a $2\times 2$ positive definite matrix, then there are vectors $\vecv(t)\in \mathbb{R}^{d+1}$ for each $t\in [0,2\pi]$ such that $F_G(\vecv(t))=Z_G(N,\vecmu).$
We will use these polynomials for two different things. First we show that there exists a $t_1$ such that the zeros of $F_G(\vecv(t_1)|z)$ lie on a circle for all $d$-regular graphs $G$. This will enable us to use a standard technique about limits of root measures. More details about this plan can be found in the introduction of Section~\ref{analysis-rank-2-approximation}. The second application is that there  is a $t_0$ such that the first coordinate of $\vecv(t_0)$ is exactly $\Phi_{d}(N,\vecmu)$ and all other coordinates have a nice sign structure. This will enable us to prove the aforementioned stability theorem for graphs containing a linear number of short cycles.  

\bigskip

\noindent \textbf{Notations.} Given a graph $G=(V,E)$ 
we use the notation $v(G)$ and $e(G)$ for the number of vertices and edges, respectively. Given a set $S\subseteq V$ let $E(S)$ denote the  edges induced by $S$, that is, $\{(u,v)\in E(G)\ |\ u,v\in S\}$ and let $e(S)=|E(S)|$. Let $G[S]$ denote the induced subgraph with vertex set $S$ and edge set $E(S)$.
Similarly, $e(V-S)$ is the number of edges induced by $V\setminus S$, and $G-S$ denotes the subgraph induced by $V\setminus S$.

For an $A\subseteq E(G)$ and a $v\in V$ let $d_A(v)$ denote the degree of the vertex $v$ in the graph $(V,A)$, that is, the number of edges of $A$ incident to $v$.

For an $A\subset E(G)$ and an $S\subseteq V(G)$ let
$A\llbracket S\rrbracket=\{(u,v)\in A\ |\ u,v\in S\}$, so these are the edges of $A$ that are induced by $S$.

Given graphs $H$ and $G$ let $\hom(H,G)$ denote the number of homomorphisms from $H$ to $G$, that is, the number of maps $\varphi:V(H)\to V(G)$ such that $(\varphi(u),\varphi(v))\in E(G)$ whenever $(u,v)\in E(H)$.

The notation $[q]$ stands for the set $\{1,2,\dots ,q\}$. We denote the scalar products of vectors $\x$ and $\y$ by $\langle \x,\y\rangle$.
\bigskip

\noindent \textbf{This paper is organized as follows.} 
In the next section we introduce the rank $1$ and rank $2$ approximations of $Z_G(q,w)$ and study its basic properties.  In Section~\ref{analysis-rank-2-approximation} we study the rank $2$ approximation or more generally, $Z_G(N,\vecmu)$. We end the paper with some remarks about the case $1<q<2$.

\section{Approximations} \label{sec: approximations}

In this section we introduce various approximations of the partition function of the random cluster model. In the sequel the rank 2 approximation will be especially important for us.

\subsection{Rank 1 approximation.}
For motivational purposes let us assume for a moment that $q$ is a positive integer. Then it is known that
$$Z_G(q,w)=Z_G(M),$$
where $M$ is the $q\times q$ matrix with entries $1+w$ in the diagonal and $1$'s as off-diagonal elements. It is a natural idea to approximate $M$ with the rank $1$ matrix $M_1$ such that the sum of all entries of $M$ and $M_1$ are equal. In other words, let $M_1$ be the $q\times q$ matrix with entries $1+\frac{w}{q}$ everywhere. Note that by the definition of $Z_G(M_1)$ we have
$$Z_G(M_1)=q^{v(G)}\left(1+\frac{w}{q}\right)^{e(G)}.$$
Let us call the quantity
$$Z^{(1)}_G(q,w)=q^{v(G)}\left(1+\frac{w}{q}\right)^{e(G)}$$
the rank $1$ approximation of $Z_G(q,w)$. This quantity makes sense even if $q$ is positive, but not necessarily integer and we will refer to it as the rank $1$ approximation of $Z_G(q,w)$ even in this case.

\begin{Lemma} \label{rank-1-inequalities}
If $q\geq 1$, then
$$Z_G(q,w)\geq Z^{(1)}_G(q,w).$$
If $0< q\leq 1$, then 
$$Z_G(q,w)\leq  Z^{(1)}_G(q,w).$$
\end{Lemma}

\begin{proof}
Using the fact that $k(A)\geq v(G)-|A|$ for an $A\subseteq E(G)$ we get that for $q\geq 1$ we have
 $$Z_G(q,w)=\sum_{A\subseteq E(G)}q^{k(A)}w^{|A|}\geq \sum_{A\subseteq E(G)}q^{v(G)-|A|}w^{|A|}=q^{v(G)}\left(1+\frac{w}{q}\right)^{e(G)}.$$
 For $q\leq 1$ we have the opposite inequality in the above computation.
 \end{proof}
 
 Lemma~\ref{rank-1-inequalities} implies that for any $d$--regular graph $G$ and $q>1$ we have 
 $$Z_G(q,w)^{1/v(G)}\geq q\left(1+\frac{w}{q}\right)^{d/2}.$$
 For $q<1$ the same quantity is an upper bound. Note that the very same quantity appears as $\Phi_{d,q,w}(0)$ in Theorem~\ref{main}.
 \bigskip

 \subsection{Rank 2 approximation} What is better than a rank $1$ approximation? Naturally, a rank $2$ approximation. 
 
 Again for motivational purposes let us assume for a moment that $q\geq 2$ is an integer. This time let us approximate the matrix $M$ with the following rank $2$ matrix $M_2$.
 $$M_2=\left(\begin{array}{cccc}
            1+w & 1 & \hdots & 1 \\
            1   & 1+\frac{w}{q-1} & \hdots & 1+\frac{w}{q-1} \\
            \vdots & \vdots & \ddots & \vdots \\
            1 & 1+\frac{w}{q-1} & \hdots & 1+\frac{w}{q-1}
            \end{array}\right).$$
 Then
 $$Z_G(M_2)=\sum_{S\subseteq V}(1+w)^{e(S)}(q-1)^{v(G)-|S|}\left(1+\frac{w}{q-1}\right)^{e(G-S)}.$$
 Indeed, let $S=\varphi^{-1}(1)$ in the definition of $Z_G(M_2)$. Let us introduce the quantity
 $$Z^{(2)}_G(q,w)=\sum_{S\subseteq V}(1+w)^{e(S)}(q-1)^{v(G)-|S|}\left(1+\frac{w}{q-1}\right)^{e(G-S)}.$$
 The definition of $Z^{(2)}_G(q,w)$ makes perfect sense if $q> 1$, but not necessarily integer and we will refer to it as the rank $2$ approximation of $Z_G(q,w)$. Recall that 
$$M'_2=\left(\begin{array}{cc}
            1+w & 1   \\
            1   & 1+\frac{w}{q-1}
            \end{array}\right)\ \ \ \text{and}\ \ \ \underline{\nu}_2=\left(\begin{array}{c}
            1   \\
            q-1
            \end{array}\right),$$ 
and note that
 $$Z^{(2)}_G(q,w)=Z_G(M'_2,\underline{\nu}_2)$$
 even if $q$ is not an integer. 
 \bigskip
 
 This time it is less clear that it is a natural approximation, but as it will turn out this is an asymptotically precise approximation for essentially large girth graphs if $q\geq 2$ and $w\geq 0$. We can prove it through a series of lemmas. 
 
 \begin{Lemma} \label{recursion}
 We have
 $$Z_G(q,w)=\sum_{S\subseteq V}(1+w)^{e(S)}Z_{G-S}(q-1,w).$$
 \end{Lemma}
 
 \begin{proof}
 This identity is trivially true for positive integer $q$ using the interpretation of $Z_G(q,w)$ as the partition function of the Potts-model.
Since we have polynomials on both sides we get that it is true for all $q$ and $w$.
 \end{proof}
 
 \begin{Lemma} \label{rank-2-inequalities}
 For $q\geq 2$ we have
 $$Z_G(q,w)\geq Z^{(2)}_G(q,w).$$
 For $1< q\leq 2$ we have
 $$Z_G(q,w)\leq Z^{(2)}_G(q,w).$$
 \end{Lemma}
 
\begin{proof} 
By Lemma~\ref{recursion} we have 
$$Z_G(q,w)=\sum_{S\subseteq V}(1+w)^{e(S)}Z_{G-S}(q-1,w).$$
By the definitions of $Z^{(2)}_G(q,w)$ and $Z^{(1)}_G(q,w)$ we have 
$$Z^{(2)}_G(q,w)=\sum_{S\subseteq V}(1+w)^{e(S)}Z^{(1)}_{G-S}(q-1,w).$$
Now the claim follows by Lemma~\ref{rank-1-inequalities}
\end{proof}

Now we are ready to prove that the rank $2$ approximation is asymptotically precise for essentially large girth graphs if $q\geq 2$ and $w\geq 0$. 

\begin{Th} \label{main approximation}
Let $G$ be a graph on $n$ vertices with $L=L(G,g)$ cycles of length at most $g-1$. Let $q\geq 2$. Then
$$Z^{(2)}_G(q,w)\leq Z_G(q,w)\leq q^{n/g+L}Z^{(2)}_G(q,w).$$
\end{Th}

\begin{proof}
The lower bound was already proven in Lemma~\ref{rank-2-inequalities}. So we only need to prove the upper bound.

Given $A\subseteq E(G)$ we can decompose $A$ as follows. Let $V_1,\dots ,V_r$ be the vertex sets of the connected components of the graph $H=(V,A)$, and let $A_1,\dots ,A_r$ be the corresponding subsets of $A$. If $V_i$ is an isolated vertex, then $A_i=\emptyset$. 

Let us say that $V_i$ is small if the induced graph $G[V_i]$ does not contain a cycle. In particular, $A_i$ does not contain a cycle either. Note that it is possible that $A_i$ does not contain a cycle, but the induced graph $G[V_i]$ contains a cycle, and so $V_i$ is not small. Let $\mathcal{S}_A$ denote the set of small $V_i$'s. We say that $V_i$ is large if it is not small, and we denote by $\mathcal{L}_A$ the set of large $V_i$'s.
Note that $|\mathcal{L}_A|\leq n/g+L$ since each large connected component has size at least $g$ or it contains a cycle of length at most $g-1$.

Finally, let us say that a vertex set $R$ is compatible with $A$ if $R$ is the union of some small $V_i$'s. Note that $R$ may be the empty set.  We denote this relation by $R\sim A$. Furthermore, let $A\llbracket R\rrbracket$ be the edges of $A$ induced by the vertex set $R$. Note that if $R\sim A$, then $A\llbracket R\rrbracket$ is a forest. On the other hand, there is no restriction on $A\llbracket V\setminus R\rrbracket.$
Figure~\ref{fig2} depicts an example for the introduced concepts.

 \begin{figure}[h!] 
    \centering
    \begin{tikzpicture}
    \node[vertex] (u1) at (1,0) {};
    \node[vertex] (u2) at (3,-1) {};
    \node[vertex] (u3) at (5,-1) {};
    \node[vertex] (u4) at (7,0) {};
    \node[vertex] (u5) at (8,1.25) {};
    \node[vertex] (u6) at (7,2.5) {};
    \node[vertex] (u7) at (5,3.5) {};
    \node[vertex] (u8) at (3,3.5) {};
    \node[vertex] (u9) at (1,2.5) {};
    \node[vertex] (u10) at (0,1.25) {};
    
    \draw (u1) -- (u2);
	\draw[line width=3pt] (u2) -- (u3);
	\draw (u3) -- (u4);
	\draw[line width=3pt] (u4) -- (u5);
	\draw (u5) -- (u6);
	\draw (u6) -- (u7);
	\draw[line width=3pt] (u7) -- (u8);
	\draw (u8) -- (u9);
	\draw[line width=3pt] (u9) -- (u10);
	\draw[line width=3pt] (u10) -- (u1);
	\draw[line width=3pt] (u1) -- (u9);
	\draw (u1) -- (u8);
	\draw (u2) -- (u8);
	\draw[line width=3pt] (u3) -- (u7);
	\draw (u3) -- (u8);
	\draw (u3) -- (u6);
	\draw (u4) -- (u6);

    \end{tikzpicture}
    \caption{A subgraph $A$ is depicted with thick edges. There are $4$ components. The edge sets with connected components of size $3$ and $4$ belong to $A_{\ell}$. The edge sets with connected components of size $1$ and $2$ belong to $A_s$. A compatible set $R$ is either the vertex set of the latter components or the empty set or the union of these two connected components.}
    \label{fig2}
\end{figure}
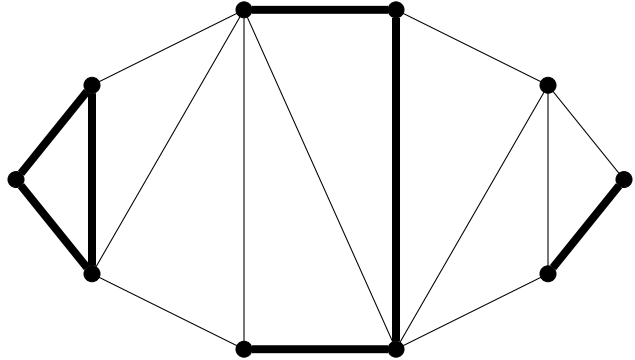

Let $k(R,A\llbracket R \rrbracket)$ denote the number of connected components of the graph $(R,A\llbracket R \rrbracket)$. By the binomial identity we have 
$$q^{|\mathcal{S}_A|}=((q-1)+1)^{|\mathcal{S}_A|}=\sum_{R\sim A}(q-1)^{k(R,A\llbracket R \rrbracket)}.$$
Then
\begin{align*}
Z_G(q,w)&=\sum_{A \subseteq E(G)}q^{k(A)}w^{|A|}\\
&=\sum_{A \subseteq E(G)}q^{|\mathcal{S}_A|+|\mathcal{L}_A|}w^{|A|}\\
&\leq q^{n/g+L}\sum_{A \subseteq E(G)}q^{|\mathcal{S}_A|}w^{|A|}\\
&=q^{n/g+L}\sum_{A \subseteq E(G)}\sum_{R: R\sim A}(q-1)^{k(R,A\llbracket R\rrbracket)}w^{|A|}\\
&=q^{n/g+L}\sum_{R \subseteq V(G)}\sum_{A: R\sim A} (q-1)^{k(R,A\llbracket R\rrbracket)}w^{|A\llbracket R\rrbracket|+|A\llbracket V\setminus R\rrbracket|}\\
&=q^{n/g+L}\sum_{R \subseteq V(G)}(1+w)^{e(V\setminus R)}\sum_{D}(q-1)^{k(R,D)}w^{|D|},
\end{align*}
where in the last sum, $D=A\llbracket R\rrbracket$ is a subset of the edges induced by $R$ such that none of the induced  connected components contains a cycle. Then
$$\sum_{D}(q-1)^{k(R,D)}w^{|D|}=\sum_{D}(q-1)^{|R|-|D|}w^{|D|}\leq (q-1)^{|R|}\left(1+\frac{w}{q-1}\right)^{e(R)}.$$
Hence 
$$Z_G(q,w)\leq q^{n/g+L}\sum_{R \subseteq V(G)}(1+w)^{e(V\setminus R)}Z^{(1)}_{G[R]}(q-1,w),$$
that is
$$Z_G(q,w)\leq q^{n/g+L}Z^{(2)}_G(q,w).$$
\end{proof}

The following theorem is an immediate consequence of Theorem~\ref{main approximation}.

\begin{Th}
Let $q\geq 2$ and $w\geq 0$. Let $(G_n)_n$ be an essentially large girth sequence of $d$-regular graphs. If the limit 
$$\lim_{n\to \infty}\frac{1}{v(G_n)}\ln Z^{(2)}_{G_n}(q,w)$$
exists, then the limit
$$\lim_{n\to \infty}\frac{1}{v(G_n)}\ln Z_{G_n}(q,w)$$
exists too, and they have the same value.
\end{Th}

\section{Ferromagnetic $2$-spin models}  \label{analysis-rank-2-approximation}

In this section we analyze the rank 2 approximation of the random cluster model. Since $Z^{(2)}_G(q,w)=Z_G(M'_2,\underline{\nu}_2)$ for a $2\times 2$ matrix $M'_2$ we will actually prove that if $N$ is a $2\times 2$ positive definite matrix with positive entries and  $\vecmu=(\mu_1,\mu_2)$ is a positive vector, then 
$$\lim_{n\to \infty}Z_{G_n}(N,\vecmu)^{1/v(G_n)}=\Phi_{d}(N,\vecmu)$$
exists for every essentially large girth sequence of $d$-regular graphs $(G_n)_n$. In fact, we will prove a much stronger theorem about Benjamini--Schramm convergent graph sequences. 

The plan is to connect the quantity $Z_G(N,\vecmu)$ with Lee-Yang theory. This connection is built out through the so-called subgraph counting polynomial (see \cite{wagner2009weighted}).

\subsection{Subgraph counting polynomial}

From now on we always assume that $G$ is a $d$-regular graph.

Let us introduce the so-called subgraph counting polynomial
$$F_G(x_0,\dots ,x_d)=\sum_{A\subseteq E}\left(\prod_{v\in V}
     x_{d_A(v)}\right),$$
and a bit more generally,
$$F_G(x_0,\dots ,x_d|z)=\sum_{A\subseteq E}\left(\prod_{v\in V}
     x_{d_A(v)}\right)z^{2|A|}=F_G(x_0,x_1z,x_2z,...,x_dz^d)$$
     As an example we give the subgraph counting polynomial $F_{K_5}(x_0,x_1,x_2,x_3,x_4)$ of the complete graph $K_5$ on $5$ vertices. The first term corresponds to the empty subgraph, the last term corresponds to the graph itself.

 \begin{align*}
 &\  x_{0}^{5} + 10 x_{0}^{3} x_{1}^{2} + 15 x_{0} x_{1}^{4} + 30 x_{0}^{2} x_{1}^{2} x_{2} + 30 x_{1}^{4} x_{2} + 60 x_{0} x_{1}^{2} x_{2}^{2} + 10 x_{0}^{2} x_{2}^{3} + 70 x_{1}^{2} x_{2}^{3} + 15 x_{0} x_{2}^{4} \\ 
 &+ 12 x_{2}^{5} + 20 x_{0} x_{1}^{3} x_{3} + 60 x_{1}^{3} x_{2} x_{3} + 60 x_{0} x_{1} x_{2}^{2} x_{3} + 120 x_{1} x_{2}^{3} x_{3} + 60 x_{1}^{2} x_{2} x_{3}^{2} + 30 x_{0} x_{2}^{2} x_{3}^{2} + 70 x_{2}^{3} x_{3}^{2} \\
 & + 60 x_{1} x_{2} x_{3}^{3} + 5 x_{0} x_{3}^{4} + 30 x_{2} x_{3}^{4} + 5 x_{1}^{4} x_{4} + 30 x_{1}^{2} x_{2}^{2} x_{4} + 15 x_{2}^{4} x_{4} + 60 x_{1} x_{2}^{2} x_{3} x_{4} + 60 x_{2}^{2} x_{3}^{2} x_{4} \\
 &+ 20 x_{1} x_{3}^{3} x_{4} + 15 x_{3}^{4} x_{4} + 10 x_{2}^{3} x_{4}^{2} + 30 x_{2} x_{3}^{2} x_{4}^{2} + 10 x_{3}^{2} x_{4}^{3} + x_{4}^{5}.
 \end{align*}
 
The general plan is the following. In the next section we show that there are vectors $\vecv(t)$ for each $t\in [0,2\pi]$ such that
$$F_G(\vecv(t))=Z_G(N,\vecmu).$$

We will show that there exists a $t_1$ such that all zeros of $F_G(\vecv(t_1)|z)$ lie on a circle for all $d$-regular graph $G$. This will imply the convergence of the sequence $\frac{1}{v(G_n)}\ln Z_G(N,\vecmu)$ for not only essentially large girth $d$-regular graphs but for all Benjamini--Schramm convergent graph sequences.

We will also show that there exists a $t_0$ such that the first coordinate of $\vecv(t_0)$ is exactly $\Phi_d(N,\vecmu)$, and all other coordinates have a nice sign structure. This will enable us to show that $Z_G(N,\vecmu)\geq \Phi_d(N,\vecmu)^{v(G)}$ for all $d$-regular graph $G$, and if $G$ contains a linear number of short cycles, then $Z_G(N,\vecmu)\geq ((1+\delta)\Phi_d(N,\vecmu))^{v(G)}$ for some $\delta>0$.

\subsection{Rank 2 matrices}
 Suppose that we can write an $r\times r$ matrix $N$ into the form $N=\veca\veca^T+\vecb\vecb^T$ and let $\vecmu\in \mathbb{R}^r$. Then
 \begin{align*}
 Z_G(N,\vecmu)&=\sum_{\varphi: V\to [r]}\prod_{v\in V}\mu_{\varphi(v)}\prod_{(u,v)\in E}N_{\varphi(u)\varphi(v)}\\
 &=\sum_{\varphi: V\to [r]}\prod_{v\in V}\mu_{\varphi(v)}\prod_{(u,v)\in E}(\veca\veca^T+\vecb\vecb^T)_{\varphi(u)\varphi(v)}\\
 &=\sum_{A\subseteq E}\sum_{\varphi: V\to [r]}\prod_{v\in V}\mu_{\varphi(v)}\prod_{(u,v)\in E\setminus A}(\veca\veca^T)_{\varphi(u)\varphi(v)}\prod_{(u,v)\in  A}(\vecb\vecb^T)_{\varphi(u)\varphi(v)}\\
 &=\sum_{A\subseteq E}\sum_{\varphi: V\to [r]}\prod_{v\in V}\mu_{\varphi(v)}\prod_{(u,v)\in E\setminus A}(\veca_{\varphi(u)}\veca_{\varphi(v)})\prod_{(u,v)\in  A}(\vecb_{\varphi(u)}\vecb_{\varphi(v)})\\
 &=\sum_{A\subseteq E}\prod_{v\in V}\left(\sum_{k=1}^r\mu_ka_k^{d-d_S(v)}b_k^{d_S(v)}\right)\\
 &=F_G(r_0,\dots ,r_d),
 \end{align*}
 where $r_j=\sum_{k=1}^r\mu_ka_k^{d-j}b_k^{j}$.
 On the other hand, $\veca$ and $\vecb$ are not the only vectors satisfying $N=\veca\veca^T+\vecb\vecb^T$. Indeed, let us define the  vectors $\veca(t)$ and $\vecb(t)$ as follows:
 $$\veca(t)_j=a_j\cos(t)+b_j\sin(t),$$
 and
 $$\vecb(t)_j=-a_j\sin(t)+b_j\cos(t).$$
 Then $N=\veca(t)\veca(t)^T+\vecb(t)\vecb(t)^T$.
 So each pair $\veca(t),\vecb(t)$ gives rise to a vector
 $\vecv(t)=(r_0(t),\dots ,r_d(t))$ such that
 $$F_G(\vecv(t))=Z_G(N,\vecmu).$$
 
 \begin{Rem} We can apply our argument to $N=M'_2$, $\vecmu=\underline{\nu}_2$ with the following vectors.
 $$\veca=\left(\begin{array}{c} \sqrt{1+\frac{w}{q}}\\ \sqrt{1+\frac{w}{q}}\end{array}\right)\ \ \ \text{and}\ \ \ \vecb=\left(\begin{array}{c}\sqrt{\frac{(q-1)w}{q}}\\
 -\sqrt{\frac{w}{q(q-1)}}\end{array}\right).$$
 One can check that $M'_2=\veca\veca^T+\vecb\vecb^T$ indeed holds true. We can again introduce the vectors $\veca(t),\vecb(t)$ giving rise to a vector
 $\vecv(t)=(r_0(t),\dots ,r_d(t))$ such that
 $$F_G(\vecv(t))=Z_G(M'_2,\underline{\nu}_2)=Z^{(2)}_G(q,w).$$
 In this case
 \begin{align*}
r_j(t)&=\sum_{k=1}^2\mu_ka(t)_k^{d-j}b(t)_k^{j}  \\
      &=\left(\sqrt{1+\frac{w}{q}}\cos(t)+\sqrt{\frac{(q-1)w}{q}}\sin(t)\right)^{d-j}\left(-\sqrt{1+\frac{w}{q}}\sin(t)+\sqrt{\frac{(q-1)w}{q}}\cos(t)\right)^{j}\\
      &\ +(q-1)\left(\sqrt{1+\frac{w}{q}}\cos(t)-\sqrt{\frac{w}{q(q-1)}}\sin(t)\right)^{d-j}\left(-\sqrt{1+\frac{w}{q}}\sin(t)-\sqrt{\frac{w}{q(q-1)}}\cos(t)\right)^{j}.
 \end{align*}
 In particular,
 $$r_0(t)=\left(\sqrt{1+\frac{w}{q}}\cos(t)+\sqrt{\frac{(q-1)w}{q}}\sin(t)\right)^{d}+(q-1)\left(\sqrt{1+\frac{w}{q}}\cos(t)-\sqrt{\frac{w}{q(q-1)}}\sin(t)\right)^{d}.$$
 In other words, $r_0(t)=\Phi_{d,q,w}(t)$.
 \end{Rem}
 
 \subsubsection{Decompositions of $2\times 2$ positive definite matrices.} Sometimes it will be convenient to require extra conditions about the vectors $\veca$ and $\vecb$ in the decomposition of $N=\veca\veca^T+\vecb\vecb^T$. 
 
 \begin{Lemma} \label{lemma:decompositions}
 Let $N$ be a $2\times 2$ positive definite matrix with positive entries. \\
 (i) Then there exists a decomposition $N=\veca\veca^T+\vecb\vecb^T$ such that $a_1,a_2,b_1>0$ and $b_2<0$.\\
 (ii) There is also a decomposition $N=\veca'\veca'^T+\vecb'\vecb'^T$ such that $a'_1,a'_2,b'_1,b'_2>0$.
 \end{Lemma}
 
 \begin{proof}
 First we prove (i). Let $\vecv_1,\vecv_2$ be the orthonormal set of eigenvectors of $N$ corresponding to eigenvectors $\lambda_1, \lambda_2>0$. If $\lambda_1\geq \lambda_2$, then by the Perron-Frobenius theory we can assume that $\vecv_1$ has positive entries. Since $\vecv_1,\vecv_2$ are orthogonal, one of the entries of $\vecv_2$ is positive, the other is negative. By considering $-\vecv_2$ if necessary we can assume that the first entry is positive, the second is negative. Hence $\veca=\sqrt{\lambda_1}\vecv_1$ and $\vecb=\sqrt{\lambda_2}\vecv_2$ satisfies the conditions.
 
 Next let us prove (ii). We can assume that we have already found an $\veca$ and $\vecb$ such that $N=\veca\veca^T+\vecb\vecb^T$ and $a_1,a_2,b_1>0$ and $b_2<0$. Let
 $$a_1'=a_1\cos(\alpha)+b_1\sin(\alpha)\ \ \ \text{and}\ \ \ b_1'=-a_1\sin(\alpha)+b_1\cos(\alpha),$$
 and 
 $$a_2'=a_2\cos(\alpha)+b_2\sin(\alpha)\ \ \ \text{and}\ \ \ b_2'=-a_2\sin(\alpha)+b_2\cos(\alpha),$$
 If we choose $\alpha$ such a way that $\alpha\in \left(-\frac{\pi}{2},\frac{\pi}{2}\right)$, that is, $\cos(\alpha)>0$ and 
 $$-\frac{a_2}{b_2}>\frac{b_1}{a_1}>0>\frac{b_2}{a_2}>\tan(\alpha)>-\frac{a_1}{b_1},$$
 then $a'_1,a'_2,b'_1,b'_2>0$. Note that $\frac{b_2}{a_2}>-\frac{a_1}{b_1}$ since $N_{12}=a_1a_2+b_1b_2>0$. 
 \end{proof}
 
 \subsubsection{The functions $a_1(t),a_2(t),b_1(t),b_2(t)$} In this section we introduce some functions that will appear many times in this paper.
 
\begin{Def}\label{def_ab}
For $a_1,a_2,b_1,b_2\in \mathbb{R}$ let
$$a_1(t)=a_1\cos(t)+b_1\sin(t)\ \ \text{and}\ \ \ b_1(t)=b_1\cos(t)-a_1\sin(t),$$
$$a_2(t)=a_2\cos(t)+b_2\sin(t)\ \ \text{and}\ \ \ b_2(t)=b_2\cos(t)-a_2\sin(t).$$
\end{Def}

\begin{Lemma} \label{equivalence}
Suppose that for the $2\times 2$ positive definite matrix $N$ we have 
 $N=\veca\veca^T+\vecb\vecb^T=\hat{\veca}\hat{\veca}^T+\hat{\vecb}\hat{\vecb}^T$. Then there exists a $t$ such that $\veca(t)=\hat{\veca}$ and $\vecb(t)=\hat{\vecb}$ or there exists a $t$ such that $\veca(t)=\hat{\veca}$ and $\vecb(t)=-\hat{\vecb}$, and all vectors of those forms are solutions.
\end{Lemma}

\begin{proof}
Consider the vectors $\x=(a_1,b_1)$ and $\y=(a_2,b_2)$. Our goal is to prove that $U(2)$ act transitively on the pairs $\x,\y$. The equation $N=\veca\veca^T+\vecb\vecb^T$ is equivalent to $N_{11}=\langle \x,\x \rangle$, $N_{12}=\langle\x,\y\rangle$, $N_{22}=\langle\y,\y\rangle$. Thus we know the length of $\x,\y$, and from these the angle between them. Thus with unitary operation we can transform any solution to any other solution, and by a unitary action applied to a solution we always get a solution.
\end{proof}

\begin{Rem}
For the specific choice $a_1=a_2=\sqrt{1+\frac{w}{q}}$, $b_1=\sqrt{\frac{(q-1)w}{q}}$ and $b_2=-\sqrt{\frac{w}{q(q-1)}}$ we use the notation
$$a_{q,w,1}(t)=\sqrt{1+\frac{w}{q}}\cos(t)+\sqrt{\frac{(q-1)w}{q}}\sin(t)\ \ \ \text{and}\ \ \ b_{q,w,1}(t)=-\sqrt{1+\frac{w}{q}}\sin(t)+\sqrt{\frac{(q-1)w}{q}}\cos(t),$$
$$a_{q,w,2}(t)=\sqrt{1+\frac{w}{q}}\cos(t)-\sqrt{\frac{w}{q(q-1)}}\sin(t)\ \ \ \text{and}\ \ \ b_{q,w,2}(t)=-\sqrt{1+\frac{w}{q}}\sin(t)-\sqrt{\frac{w}{q(q-1)}}\cos(t).$$
\end{Rem}

We collected some claims about $a_1(t),a_2(t),b_1(t),b_2(t)$ whose proof is just a straightforward computation.
First we describe the sign structure of the functions $a_1(t),a_2(t),b_1(t),b_2(t)$ on the interval $\left[0,\frac{\pi}{2}\right)$.

\begin{Lemma} Let $a_1,a_2,b_1,b_2\in \mathbb{R}$ such that $a_1,a_2,b_1>0$ and $b_2<0$ and $a_1a_2+b_1b_2>0$. 
Let  $t\in \left[0,\frac{\pi}{2}\right)$. Then\\
(a) if $0\leq \tan(t)\leq \frac{b_1}{a_1}$ we have $a_1(t),a_2(t),b_1(t)\geq 0$ and $b_2(t)<0$,\\
(b) if $\frac{b_1}{a_1}\leq \tan(t)\leq \frac{a_2}{-b_2}$ we have 
 $a_1(t),a_2(t)\geq 0$ and $b_1(t),b_2(t)\leq 0$,\\
(c) if $\frac{a_2}{-b_2}\leq \tan(t)$, then $a_1(t)>0$ and $a_2(t),b_1(t),b_2(t)\leq 0$
\end{Lemma}

\begin{Lemma} Let $a_1,a_2,b_1,b_2\in \mathbb{R}$. 
Then 
$$\frac{\partial}{\partial t}\left(\frac{a_1(t)b_1(t)}{a_2(t)b_2(t)}\right)=\frac{(a_1a_2+b_1b_2)(a_2b_1-a_1b_2)}{a_2(t)^2b_2(t)^2}.$$
\end{Lemma}

\begin{Lemma} \label{all-value}
Let $Q$ be a $2\times 2$ real matrix with non-zero determinant.  Then for every $c\in \mathbb{R}$, there is a unique  $t\in [0,\pi)$ such that $\frac{Q_{11}\cos(t)+Q_{12}\sin(t)}{Q_{21}\cos(t)+Q_{22}\sin(t)}=c$.
\end{Lemma}

\begin{proof}
Let
$F_Q(t)=\frac{Q_{11}\cos(t)+Q_{12}\sin(t)}{Q_{21}\cos(t)+Q_{22}\sin(t)}.$
We have 
$\frac{\partial}{\partial t}F_Q(t)=\frac{Q_{12}Q_{21}-Q_{11}Q_{22}}{(Q_{21}\cos(t)+Q_{22}\sin(t))^2}.$
Hence $F_Q(t)$ is either strictly monotone decreasing or strictly monotone increasing on $[0,\pi)$ depending on the sign of $\det(Q)$ with a discontinuity at $t_0$, where $\tan(t_0)=-\frac{Q_{21}}{Q_{22}}$. Since $F_Q(0)=F_Q(\pi)=\frac{Q_{11}}{Q_{21}}$, and 
$\lim_{t\searrow t_0}F_Q(t)=\pm \infty$ and  $\lim_{t\nearrow t_0}F_Q(t)=\mp \infty$, the claim follows.
\end{proof}

We will also use the following identities. 
 
\begin{Lemma} \label{identities}
For arbitrary $a_1,a_2,b_1,b_2\in \mathbb{R}$ we have
$$a_1(t)^2+b_1(t)^2=a_1^2+b_1^2,\  \ \ 
a_2(t)^2+b_2(t)^2=a_2^2+b_2^2,$$
$$a_1(t)a_2(t)+b_1(t)b_2(t)=a_1a_2+b_1b_2,\ \ \ a_2(t)b_1(t)-a_1(t)b_2(t)=a_2b_1-a_1b_2.$$
\end{Lemma}

\begin{Rem}
In case of $Z^{(2)}_G(q,w)$ we get
$$a_{q,w,1}(t)a_{q,w,2}(t)+b_{q,w,1}(t)b_{q,w,2}(t)=1.$$
It is also true that
$$a_{q,w,1}(t)b_{q,w,1}(t)+(q-1)a_{q,w,2}(t)b_{q,w,2}(t)=-q\cos(t)\sin(t).$$
\end{Rem}

\subsection{The vector $\vecv(t_1)$} In this section we show that there exists a $t_1$ such that for all $d$-regular graph $G$ the zeros of $F_G(\vecv(t_1)|z)$ lie on a circle.

\subsubsection{Wagner's subgraph counting technique}

In this section we will recall some theorem of Wagner (Theorem~3.2 of \cite{wagner2009weighted}) about the location of zeros of $F_G(x_0,\dots,x_d|z)$.  For any fixed $x_0,\dots,x_d$ let us define the following \emph{key-polynomial}
\[
    K(x_0,\dots,x_d|z)=\sum_{k=0}^d {d\choose k} x_k z^k.
\]

\begin{Th}[Wagner \cite{wagner2009weighted}]\label{thm:wagner}
If $K(x_0,\dots,x_d|z)$ has no complex zero in the open disk of radius $\kappa$ around 0, then $F_G(x_0,\dots,x_d |z)$ has no complex zero in the open disk of radius $\kappa$ around 0 for any $d$-regular graph $G$.

If $K(x_0,\dots,x_d|z)$ has no complex zero in the complement of a closed disk of radius $\kappa$ around 0, then $F_G(x_0,\dots,x_d |z)$ has no complex zero in the complement of a closed disk of radius $\kappa$ around 0 for any $d$-regular graph $G$.

In particular, if  $K(x_0,\dots,x_d|z)$ has only zeros on the circle of radius $\kappa$ around 0, then $F_G(x_0,\dots,x_d|z)$ has complex zeros only on the circle of radius $\kappa$ for any $d$-regular graph $G$.
\end{Th}

\subsubsection{Key polynomials for rank 2 matrices}

Suppose that we have a rank 2 matrix $N$ of the form $N=\veca\veca^T+\vecb\vecb^T\in \mathbb{R}$ and a $\vecmu\in\mathbb{R}^{2}$. Then we know that $F_G(\vecv(t))=Z_G(N,\vecmu)$, where $\vecv(t)=(r_0(t),\dots,r_d(t))$ for any $t\in[0,2\pi)$. 

\begin{Lemma} \label{key-rank-2}
Let $\veca,\vecb,\underline{\mu}\in\mathbb{R}^{r}$. For $k=1,\dots ,r$ let
$$a_k(t)=a_k\cos(t)+b_k\sin(t)\ \ \ \text{and}\ \ \ b_k(t)=b_k\cos(t)-a_k\sin(t),$$
and for $j=0,\dots ,d$ let $r_j(t)=\sum_{k=1}^r\mu_ka_k(t)^{d-j}b_k(t)^j$. Finally, let $\vecv(t)=(r_0(t),\dots ,r_d(t))$ and
$$K(\vecv(t) | z)=\sum_{j=0}^d\binom{d}{j}r_j(t)z^j.$$
Then 
\[
    K(\vecv(t) | z)= \sum_{k=1}^r \mu_k(b_k(t)z+a_k(t))^d
\]
\end{Lemma}
\begin{proof}
By definition we have 
\begin{align*}
    K(\vecv(t) | z)&=\sum_{j=0}^d \binom{d}{j}r_j(t)z^j \\
    &=\sum_{j=0}^d \binom{d}{j}\left(\sum_{k=1}^{r} \mu_ka_k(t)^{d-j}b_k(t)^j\right) z^j\\
    &=\sum_{k=1}^{r} \sum_{j=0}^d \binom{d}{j} \mu_ka_k(t)^{d-j}b_k(t)^jz^j\\
    &=\sum_{k=1}^{r} \mu_k\left(a_k(t)+b_k(t)z\right)^d.
\end{align*}
\end{proof}

\begin{Lemma}\label{lemma:key_zeros}
Let $\mu_1,\mu_2\in\mathbb{R}$ and $a_1,a_2,b_1,b_2\in \mathbb{R}$ such that $\mu_1,\mu_2> 0$, then all the complex zeros of $K(\vecv (t)|z)$ are on a circle or on a line. 

Moreover, if $t_1$ satisfies

\[
    \frac{a_1(t_1)b_1(t_1)}{a_2(t_1)b_2(t_1)}=\left(\frac{\mu_2}{\mu_1}\right)^{2/d},
\]
then the circle has center at $0$. Furthermore, the radius of this circle is

\[
    R_c=\left(\frac{\mu_2}{\mu_1}\right)^{1/d}\left|\frac{a_2(t_1)}{b_1(t_1)}\right|=\left(\frac{\mu_2}{\mu_1}\right)^{-1/d}\left|\frac{a_1(t_1)}{b_2(t_1)}\right|=\left|\frac{a_1(t_1)a_2(t_1)}{b_1(t_1)b_2(t_1)}\right|^{1/2}.
\]
\end{Lemma}
\begin{proof}
From Lemma~\ref{key-rank-2} we have that
\[
    K(\vecv(t)|z)=\mu_1(a_1(t)+b_1(t) z)^d+\mu_2(a_2(t)+b_2(t) z)^d.
\]
Let us assume that $K(\vecv(t)|\zeta)=0$. 

If $a_1(t)+b_1(t)\zeta=a_2(t)+b_2(t)\zeta=0$, then $\zeta$ is the only zero of $K(\vecv(t)|z)$ with multiplicity $d$, thus all the complex zeros are on a circle of radius $|\zeta|=\left|\frac{a_1(t)a_2(t)}{b_1(t)b_2(t)}\right|^{1/2}$ with center at 0. 

If $a_1(t)+b_1(t)\zeta$ or $a_2(t)+b_2(t)\zeta$ is not 0, then by symmetry we can assume that $a_2(t)+b_2(t)\zeta\neq 0$, and
we get that
\begin{align*}
    \mu_1(a_1(t)+b_1(t) \zeta)^d+\mu_2(a_2(t)+b_2(t) \zeta)^d&=0\\
    \left(\frac{a_1(t)+b_1(t)\zeta}{a_2(t)+b_2(t)\zeta}\right)^d&=-\frac{\mu_2}{\mu_1}\\
    M_t(\zeta)^d&=-\frac{\mu_2}{\mu_1},
\end{align*}
where $M_t(z)=\frac{a_1(t)+b_1(t)z}{a_2(t)+b_2(t)z}$ is a M\"obius transformation with real coefficients. Let us introduce the notation $T=\left(\frac{\mu_2}{\mu_1}\right)^{2/d}$. Thus we obtained that for any $\zeta$ zero of $K(\vecv(t)|z)$ we have
\[
    |M_t(\zeta)|=\sqrt{T}.
\]
Since $M_t(z)$ is M\"obius transformation, therefore $M_t^{(-1)}(z)$ maps cycles into cycles and lines, i.e. $\zeta\in M_{t}^{(-1)}(S_{\sqrt{T}} )$, where $S_c$ is a circle of radius $c$ around $0$.

In order to prove the second part of the statement we have to investigate when does the circle $M_{t_1}^{(-1)}(S_{\sqrt{T}})$ have a center at 0. Since $M_t(z)$ is M\"obius transformation with real coefficients, thus $M_t^{(-1)}(z)$ is also a M\"obius transformation with real coefficients. This means that the image of a circle that is perpendicular to the real line is also perpendicular to the real line. We claim that $M_{t_1}^{(-1)}(S_{\sqrt{T}})$ is not a line. To see it it is enough to show that $M^{(-1)}_{t_1}(\pm \sqrt{T})$ is not $\infty$, or equivalently $M_{t_1}(\infty)\neq \pm\sqrt{T}$. If this would be the case, then $M_{t_1}(\infty)=\frac{b_1(t_1)}{b_2(t_1)}=\pm\sqrt{T}$ would imply that $a_1(t_1)+b_1(t_1)z$ and $a_2(t_1)+b_2(t_1)z$ have a common zero, which lead us to a contradiction.

Thus the center of $M^{(-1)}_{t_!}(S_{\sqrt{T}} )$ is at 
\[
    \frac{1}{2}\left(M^{(-1)}_{t_1}\left(\sqrt{T}\right)+M^{(-1)}_{t_1}\left(-\sqrt{T}\right)\right).
\]
This is $0$ if and only if %$t=t_1$ satisfies
\begin{align*}
    M^{(-1)}_{t_1}\left(\sqrt{T}\right)&=-M^{(-1)}_{t_1}\left(-\sqrt{T}\right)\\
    \frac{a_2(t_1) \sqrt{T}-a_1(t_1)}{-b_2(t_1)\sqrt{T} +b_1(t_1)}&= \frac{a_2(t_1) \sqrt{T}+a_1(t_1)}{b_2(t_1)\sqrt{T} +b_1(t_1)}\\
    a_2(t_1)b_2(t_1)T&= a_1(t_1)b_1(t_1)
\end{align*}
This is equivalent to 
\[
    T= \frac{a_1(t_1)b_1(t_1)}{a_2(t_1)b_2(t_1)}.
\]
To find the corresponding radius we have to calculate $\left|M_{t_1}^{(-1)}(\sqrt{T})\right|$.% for the corresponding $t_1$.
\begin{align*}
    M_{t_1}^{(-1)}(\sqrt{T})&=\frac{a_2(t_1) \sqrt{T}-a_1(t_1)}{-b_2(t_1)\sqrt{T} +b_1(t_1)}\\
    &=\frac{a_2(t_1)}{b_1(t_1)}\sqrt{T}\left(\frac{a_2(t_1)b_1(t_1)\sqrt{T}-a_1(t_1)b_1(t_1)} {-a_2(t_1)b_2(t_1)T+b_1(t_1)a_2(t_1)\sqrt{T}}\right)\\
    &=\frac{a_2(t_1)}{b_1(t_1)}\sqrt{T}
\end{align*}
This implies that $R_c=\sqrt{T}\left|\frac{a_2(t_1)}{b_1(t_1)}\right|$. Thus by equation $T= \frac{a_1(t_1)b_1(t_1)}{a_2(t_1)b_2(t_1)}$ we also have $R_c=T^{-1/2}\left|\frac{a_1(t_1)}{b_2(t_1)}\right|$, and by multiplying the two equations we get that 
$R_c^2=\left|\frac{a_1(t_1)a_2(t_1)}{b_1(t_1)b_2(t_1)}\right|$.
\end{proof}

\begin{Lemma}\label{lemma:key_rotation_exists}
Let $a_1,a_2,b_1,b_2,\mu_1,\mu_2\in \mathbb{R}$ such that $a_1,a_2,b_1,\mu_1,\mu_2>0$ and $b_2<0$ and $a_1a_2+b_1b_2>0$, then there is a unique $t_1\in \left[0,\frac{\pi}{2}\right]$ such that $\frac{b_1}{a_1}<\tan(t_1)<\frac{a_2}{-b_2}$ and 
$$\frac{a_1(t_1)b_1(t_1)}{a_2(t_1)b_2(t_1)}=\left(\frac{\mu_2}{\mu_1}\right)^{2/d}.$$
For such a $t_1$ we have $a_1(t_1),a_2(t_1)>0$ and $b_1(t_1),b_2(t_1)<0$ implying that $\frac{a_1(t_1)a_2(t_1)}{b_1(t_1)b_2(t_1)}>0$.
\end{Lemma}

\begin{proof}
Note that the function $\frac{a_1(t)b_1(t)}{a_2(t)b_2(t)}$ is only positive at $t\in \left[0,\frac{\pi}{2}\right]$ if $a_1(t),a_2(t)>0$ and $b_1(t),b_2(t)<0$, that is, $\frac{b_1}{a_1}<\tan(t)<\frac{a_2}{-b_2}$.
When $t\to \arctan\left(\frac{b_1}{a_1}\right)$, then $b_1(t)\to 0$, and so $\frac{a_1(t)b_1(t)}{a_2(t)b_2(t)}\to 0$. If $t\to \arctan\left(\frac{a_2}{-b_2}\right)$, then $a_2(t)\to 0$, and so $\frac{a_1(t)b_1(t)}{a_2(t)b_2(t)}\to \infty$. Since
$$\frac{\partial}{\partial t}\left(\frac{a_1(t)b_1(t)}{a_2(t)b_2(t)}\right)=\frac{(a_1a_2+b_1b_2)(a_2b_1-a_1b_2)}{a_2(t)^2b_2(t)^2}>0$$
the function is strictly monotone increasing, hence there is a unique $t_1$ satisfying $\frac{a_1(t_1)b_1(t_1)}{a_2(t_1)b_2(t_1)}=\left(\frac{\mu_2}{\mu_1}\right)^{2/d}$.

\end{proof}

\begin{Th} \label{zeros-on-circle}
Let $N$ be a $2\times 2$ positive definite matrix with positive entries and let $\vecmu\in \mathbb{R}^2_{>0}$. 
Then there exists a $\vecv_c\in \mathbb{R}^{d+1}$ and an $R_c(N,\vecmu)\in \mathbb{R}_{>0}$ such that for any $d$-regular graph $G$ we have $Z_G(N,\vecmu)=F_G(\vecv_c)$ and  all complex zeros of $F_G(\vecv_c|z)$ lie on a circle around $0$ of radius $R_c(N,\vecmu)$. Moreover, $R=R_c(N,\mu)$ is a positive real solution of
\[
 (N_{11}N_{22} - N_{12}^2) R^4 + ( -N_{22}^2T  +  2N_{12}^2  - N_{11}^2T^{-1} ) R^2 + (N_{11}N_{22}- N_{12}^2)=0,
\]
where $T=\left(\frac{\mu_2}{\mu_1}\right)^{2/d}$.
\end{Th}

\begin{proof}
The first part of the claim follows from combining Lemma~\ref{lemma:decompositions}, ~\ref{lemma:key_rotation_exists}, \ref{lemma:key_zeros} and Theorem~\ref{thm:wagner}. Indeed, by Lemma~\ref{lemma:decompositions} we know that there exists a decomposition $N=\veca\veca^T+\vecb\vecb^T$ such that $a_1,a_2,b_1>0$ and $b_2<0$. Then Lemma~\ref{lemma:key_rotation_exists} implies that there exists a $t_1$ such that $\frac{a_1(t_1)b_1(t_1)}{a_2(t_1)b_2(t_1)}=\left(\frac{\mu_2}{\mu_1}\right)^{2/d}$. Then Lemma~\ref{lemma:key_zeros} shows that the zeros of $K(\vecv(t_1)|z)$ lie on a circle that has center at $0$. Then Theorem~\ref{thm:wagner} implies that  all  complex zeros of $F_G(\vecv(t_1)|z)$ lie on a circle around $0$ for any $d$-regular graph $G$.
Thus $\vecv_c=\vecv(t_1)$ satisfies the conditions of the theorem.

To prove the statement concerning the radius of the circle note that by Lemma~\ref{identities} and \ref{lemma:key_zeros} we have 
$N_{11}=a_1(t_1)^2+b_1(t_1)^2$, $N_{22}=a_2(t_1)^2+b_2(t_1)^2$, $N_{12}=a_1(t_1)a_2(t_1)+b_1(t_1)b_2(t_1)$, $T=\frac{a_1(t_1)b_1(t_1)}{a_2(t_1)b_2(t_1)}$ and $R^2=\frac{a_1(t_1)a_2(t_1)}{b_1(t_1)b_2(t_1)}$.
Let us introduce the notations $\ova_1=a_1(t_1)$, $\ova_2=a_2(t_1)$, $\ovb_1=b_1(t_1)$ and $\ovb_2=b_2(t_1)$.
Then we get that
\begin{align*}
&(N_{11}N_{22} - N_{12}^2) R^4 + ( -N_{22}^2T  +  2N_{12}^2  - N_{11}^2T^{-1} ) R^2 + (N_{11}N_{22}- N_{12}^2)\\
=&((\ova_1^2+\ovb_1^2)(\ova_2^2+\ovb_2^2)-(\ova_1\ova_2+\ovb_1\ovb_2)^2)\left(\frac{\ova_1^2\ova_2^2}{\ovb_1^2\ovb_2^2}+1\right)\\
&-(\ova_2^2+\ovb_2^2)^2\frac{\ova_1\ovb_1}{\ova_2\ovb_2}\cdot \frac{\ova_1\ova_2}{\ovb_1\ovb_2}
+2(\ova_1\ova_2+\ovb_1\ovb_2)^2\cdot \frac{\ova_1\ova_2}{\ovb_1\ovb_2}
-(\ova_1^2+\ovb_1^2)^2\frac{\ova_2\ovb_2}{\ova_1\ovb_1}\cdot \frac{\ova_1\ova_2}{\ovb_1\ovb_2}\\
=&\frac{(\ova_1^2\ovb_2^2-2\ova_1\ova_2\ovb_1\ovb_2+\ova_2^2\ovb_1^2)(\ova_1^2\ova_2^2+\ovb_1^2\ovb_2^2)}{\ovb_1^2\ovb_2^2}
-\frac{(\ova_2^4+2\ova_2^2\ovb_2^2+\ovb_2^4)\ova_1^2\ovb_2^2}{\ovb_1^2\ovb_2^2}\\
&+\frac{2(\ova_1^2\ova_2^2+2\ova_1\ova_2\ovb_1\ovb_2+\ovb_1^2\ovb_2^2)\ova_1\ova_2\ovb_1\ovb_2}{\ovb_1^2\ovb_2^2}
-\frac{(\ova_1^4+2\ova_1^2\ovb_1^2+\ovb_1^4)\ova_2^2\ovb_2^2}{\ovb_1^2\ovb_2^2}
\end{align*}
Now one can see that everything cancels, and this is indeed $0$.

\end{proof}

\subsection{Random regular graphs and  Bethe approximation}
\label{factor-graphs-Bethe-approximation}

In this section we recall about some results of Dembo, Montanari, Sly and Sun \cite{dembo2014replica} on Bethe approximation. We introduce a quantity $\Phi_d(N,\vecmu)$ for which it is true that if $G$ is a random $d$-regular graph on $n$ vertices, then  we have $\mathbb{E}Z_G(N,\vecmu)=n^{O(1)}\Phi_{d}(N,\vecmu)^n$. 
As a consequence of a theorem of Ruozzi we will also get that $Z_G(N,\vecmu)\geq  \Phi_{d}(N,\vecmu)^{v(G)}$ for a $d$-regular graph $G$ if $N$ is a $2\times 2$ positive definite matrix.

In general, let $N\in \mathbb{R}^{r\times r}_{>0}$ be a symmetric matrix and $\vecmu\in \mathbb{R}^r_{>0}$. Let $B_{N,\vecmu}$ be a symmetric distribution on $[r]^2$. Let $b_{N,\vecmu}$ be the marginal of $B_{N,\vecmu}$ to its first coordinate. 
Let us define
$$\mathbb{F}_{\hom}(B_{N,\vecmu}):=\frac{d}{2}H(B_{N,\vecmu})-(d-1)H(b_{N,\vecmu})+[\ln N]_{B_{N,\vecmu}}+\frac{d}{2}[\ln \mu]_{b_{N,\mu}},$$
where for a probability distribution $P=(p_1,\dots ,p_n)$ and a vector $f=(f_1,\dots ,f_n)$ we have
$$H(P)=\sum_{i=1}^np_i\ln \frac{1}{p_i}\ \ \ 
\text{and}\ \ \ 
[f]_P=\sum_{i=1}^np_if_i$$
with the usual convention $0\cdot \ln \frac{1}{0}=0$.
The subscript $\hom$ in $\mathbb{F}_{\hom}$ simply stands for homomorphism.
Let
$$\Phi_d(N,\vecmu)=\max_{B_{N,\vecmu}}\exp\left(\mathbb{F}_{\hom}(B_{N,\vecmu})\right).$$
The quantity $\Phi_d(N,\vecmu)$ also has a description through belief propagation equation, see Proposition 14.6 of \cite{mezard2009information} or Section 1.2 of \cite{dembo2014replica}. Let $h\in \mathbb{R}^r$ be a probability distribution. The belief propagation equation or Bethe recursion is
$$\mathrm{BP}(h)_{\sigma}:=\frac{1}{z_h}\mu_{\sigma}\left(\sum_{\sigma'}N_{\sigma,\sigma'}h_{\sigma'}\right)^{d-1}$$
for all $\sigma \in [r]$, where $z_h$ is the normalizing constant ensuring that $\mathrm{BP}(h)$ is a probability distribution too. Let $\mathcal{H}^*$ be the set of probability distributions for which $\mathrm{BP}(h)=h$. 
The Bethe functional is defined as 
$$\widetilde{\Phi}_{N,\vecmu,d}(h)=\ln\left(\sum_{\sigma}\mu_{\sigma}\left(\sum_{\sigma'}N_{\sigma,\sigma'}h_{\sigma'}\right)^{d}\right)-\frac{d}{2}\ln\left(\sum_{\sigma,\sigma'}N_{\sigma,\sigma'}h_{\sigma}h_{\sigma'}\right).$$
Then
$$\Phi_d(N,\vecmu)=\sup_{h\in \mathcal{H}^*}\exp(\widetilde{\Phi}_{N,\vecmu,d}(h)).$$
Dembo, Montanari, Sly and Sun \cite{dembo2014replica} showed that the quantity $\Phi_d(N,\vecmu)$ is directly related to the expected value of $Z_G(N,\vecmu)$ for random $d$-regular graphs.

\begin{Th}[Dembo, Montanari, Sly, Sun \cite{dembo2014replica}]\label{expected_value}
Let $G$ be a random $d$-regular graph on $n$ vertices. Then
$$\mathbb{E} Z_G(N,\vecmu)=n^{O(1)}\Phi_{d}(N,\vecmu)^{n}.$$
\end{Th}

\subsubsection{A theorem of Ruozzi}
In this section we show that if $N$ is a $2\times 2$ positive definite matrix with positive entries and $\vecmu$ is a positive vector, then for any $d$-regular graph $G$ we have $Z_G(N,\vecmu)\geq \Phi_{d}(N,\vecmu)^{n}$. First we recall the setting of factor graphs.

\begin{Def}
A factor graph $\mathcal{G}=(F,V,E,\mathcal{X},(g_a)_{a\in F})$ is a bipartite graph equipped  with a set of functions. Its  vertex set is $F\cup V$, where $F$ is the set of function nodes, and $V$ is the set of variable nodes. The edge set of $\mathcal{G}$ will be denoted by $E(\mathcal{G})$. The neighbors of a factor node $a$ or variable node $v$ will be denoted by $\partial a$ or $\partial v$, respectively. For each variable 
node $v$ we associate a variable $x_v$ taking its values from the alphabet $\mathcal{X}$. For each $a$ there is an associated function $g_a: \mathcal{X}^{\partial a}\to \mathbb{R}_{\geq 0}$. The partition function of the factor graph $\mathcal{G}$ is
$$Z(\mathcal{G})=\sum_{\x \in \mathcal{X}^V}\prod_{a\in F}g_a(\x_{\partial a}),$$
where $\x_{\partial a}$ is the restriction of $\x$ to the set $\partial a$.

When $\mathcal{X}=\{0,1\}$ we speak about a binary factor graph.
\end{Def}

Let us consider an example.

\begin{Ex} \label{ex: homomorphisms}
Suppose that $G=(V,E)$ is an (ordinary) graph. We can associate a factor graph $\mathcal{G}$ as follows. For each $v\in V$ we introduce a variable node $v$ and function node $v'$, and for each edge $e=(u,v)$ we introduce a function node $e$. In $\mathcal{G}$ let us connect $v$ with $v'$ and $e=(u,v)$ with $u$ and $v$. Set $\mathcal{X}=[r]$. Let $N$ be a $r\times r$ matrix and $\vecmu\in \mathbb{R}^r$. For each function node $v'$ we introduce the function $g_v(x)=\mu_x$ and for each edge $e$ we introduce the function $g_e(x,y)=N_{x,y}$. Then 
$$Z(\mathcal{G})=Z_G(N,\vecmu).$$
The middle picture of Figure~\ref{fig3} depicts the factor graph $\mathcal{G}$  for the diamond graph $G$.
\end{Ex}

\begin{Ex} \label{ex: subgraph counting} Let $G=(V,E)$ be a graph. Recall that
$$F_G(x_0,\dots ,x_d)=\sum_{A\subseteq E}\left(\prod_{v\in V} x_{d_A(v)}\right).$$
For $(a_0,\dots ,a_d)\in \mathbb{R}^{d+1}$ let us consider the following factor graph $\mathcal{G}'=(F',V',E',\mathcal{X}',(g'_a)_{a\in F}).$
We subdivide each edge of $E$ with one vertex. In the resulting bipartite graph one side corresponds to $F'$, the other side corresponds to $V'$. So with a slight abuse of notation we have $F'=V$ and $V'=E$. Let $\mathcal{X}'=\{0,1\}$. For each $v\in V$ let us introduce the function
$$g'_v(x_{e_1},\dots ,x_{e_{d_v}})=a_{|x|},$$
where $|x|=x_{e_1}+\dots +x_{e_{d_v}}$, and $e_1,\dots ,e_{d_v}$ are the edges incident to $v$. Then
$$Z(\mathcal{G}')=F_G(a_0,\dots ,a_d).$$
As we can see $\mathcal{G}'$ is in some sense the dual of $\mathcal{G}$. The picture on the right hand side of Figure~\ref{fig3} depicts the factor graph $\mathcal{G}'$ for the diamond graph $G$.
\end{Ex}

\begin{figure}[h!] 
    \centering
    \begin{tikzpicture}[scale=0.85]
    \node[vertex] (u1) at (0,0) {};
	\node[vertex] (u2) at (2,2) {};
	\node[vertex] (u3) at (2,-2) {};
	\node[vertex] (u4) at (4,0) {};
	\draw (u1) -- (u2) ;
	\draw (u1) -- (u3) ;
	\draw (u2) -- (u3) ;
    \draw (u2) -- (u4) ;
    \draw (u3) -- (u4) ;
		
	\node[Cvertex] (v1) at (7,0) {};
	\node[Cvertex] (v2) at (9,2) {};
	\node[Cvertex] (v3) at (9,-2) {};
	\node[Cvertex] (v4) at (11,0) {};
	\node[Sqvertex] (v5) at (6,0) {$\mu$};
	\node[Sqvertex] (v6) at (9,3) {$\mu$};
	\node[Sqvertex] (v7) at (9,-3) {$\mu$};
	\node[Sqvertex] (v8) at (12,0) {$\mu$};
	\node[Sqvertex] (v9) at (8,1) {$N$};
	\node[Sqvertex] (v10) at (8,-1) {$N$};
	\node[Sqvertex] (v11) at (9,0) {$N$};
	\node[Sqvertex] (v12) at (10,1) {$N$};
	\node[Sqvertex] (v13) at (10,-1) {$N$};
	
	\draw (v1) -- (v9) ;
	\draw (v1) -- (v10) ;
	\draw (v2) -- (v9) ;
	\draw (v2) -- (v11) ;
    \draw (v2) -- (v12) ;
    \draw (v3) -- (v10) ;
    \draw (v3) -- (v11) ;
    \draw (v3) -- (v13) ;
    \draw (v4) -- (v12) ;
    \draw (v4) -- (v13) ;
    \draw (v1) -- (v5) ;
    \draw (v2) -- (v6) ;
    \draw (v3) -- (v7) ;
    \draw (v4) -- (v8) ;
    
    \node[Sqvertex] (w1) at (14,0) {};
	\node[Sqvertex] (w2) at (16,2) {};
	\node[Sqvertex] (w3) at (16,-2) {};
	\node[Sqvertex] (w4) at (18,0) {};
	
	\node[Cvertex] (w9) at (15,1) {};
	\node[Cvertex] (w10) at (15,-1) {};
	\node[Cvertex] (w11) at (16,0) {};
	\node[Cvertex] (w12) at (17,1) {};
	\node[Cvertex] (w13) at (17,-1) {};
	
	\draw (w1) -- (w9) ;
	\draw (w1) -- (w10) ;
	\draw (w2) -- (w9) ;
	\draw (w2) -- (w11) ;
    \draw (w2) -- (w12) ;
    \draw (w3) -- (w10) ;
    \draw (w3) -- (w11) ;
    \draw (w3) -- (w13) ;
    \draw (w4) -- (w12) ;
    \draw (w4) -- (w13) ;

		\end{tikzpicture}
		\caption{A graph with two graphical models. Square shape nodes are function vertices, circle shape nodes are variable vertices. We can see the factor graph of Example~\ref{ex: homomorphisms} in the middle and the factor graph of Example~\ref{ex: subgraph counting} on the right. In this example the original graph is not regular, nevertheless we can see that in $\mathcal{G}'$ the variable nodes correspond to  the edges of the original graph while the function nodes correspond to  the vertices of the original graph.}
		\label{fig3}
\end{figure}
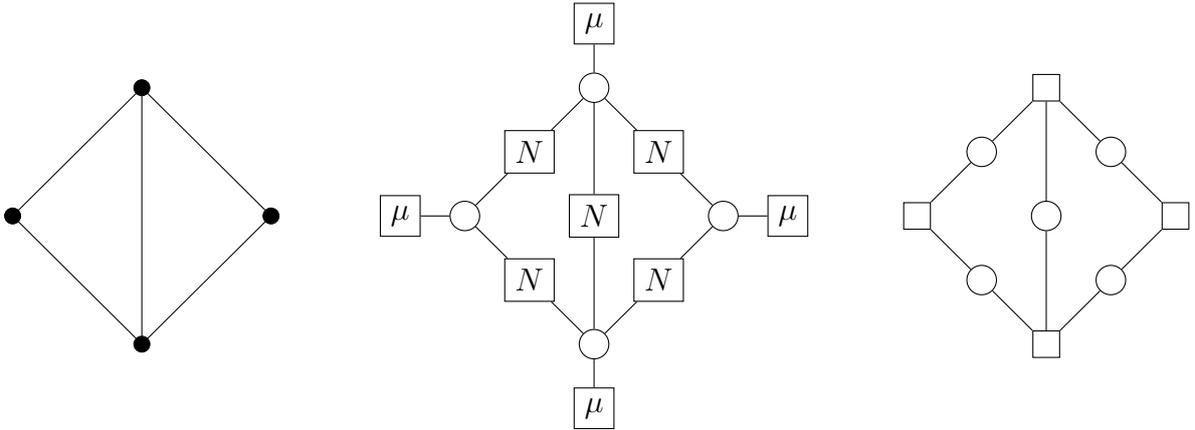

Next we need the concept of the Bethe approximation for factor graphs.  First we need to introduce the pseudo-marginal polytope. 

\begin{Def} For each variable node $v$ let us introduce a probability distribution $b_v$ on $\mathcal{X}$, and for each function node $a$ let us also introduce
a probability distribution $b_a$ on $\mathcal{X}^{\partial a}$:
$$\sum_{x\in \mathcal{X}}b_v(x)=1\ \ \forall v\in V,\ \ b_v(x)\geq 0\ \ \ \forall x\in \mathcal{X},$$
and
$$\sum_{\x\in \mathcal{X}^{\partial{a}}}b_a(\x)=1\ \ \forall a\in F,\ \ b_a(\x)\geq 0\ \ \ \forall \x\in \mathcal{X}^{\partial a}.$$
Furthermore, $b_v$ and $b_a$ have to be consistent in the following sense: for all $c\in \mathcal{X},\ a\in F, v\in \partial a$ we have
$$\sum_{\x \in \mathcal{X}^{\partial a\setminus v}}b_a(\x,c)=b_v(c).$$
We will call a $\underline{b}=((b_v)_{v\in V},(b_a)_{a\in F})$ a locally consistent set of marginals or simply pseudo-marginal. The set of such $\underline{b}$ will be denoted by $\mathrm{Mar}(\mathcal{G})$.
\end{Def}

\begin{Def} The Bethe partition function $Z_B(\mathcal{G})$ is defined as follows.
Let $\mathbb{F}$ be the following function evaluated on a $\underline{b} \in \mathrm{Mar}(\mathcal{G})$:
$$
\mathbb{F}(\underline{b})=\sum_{a\in F}\sum_{\x \in \mathcal{X}^{\partial{a}}}b_a(\x) \ln \frac{g_a(\x)}{b_a(\x)} 
                          -\sum_{v \in V}(1-|\partial v|)\sum_{x \in \mathcal{X}}b_v(x)\ln b_v(x).
$$
The notation $\mathbb{F}$ is consistent with our previous notation $\mathbb{F}_{\hom}$ as it will be  explained later. 
Finally, let
$$H_B(\mathcal{G})=\sup_{\underline{b}\in \mathrm{Mar}(\mathcal{G})} \mathbb{F}(\underline{b}),$$
and
$$Z_B(\mathcal{G})=\exp(H_B(\mathcal{G})).$$
Here $H_B(\mathcal{G})$ is the Bethe free entropy, and $Z_B(\mathcal{G})$ is the Bethe partition function. We note that if $g_a(\x)=0$, then we require $b_a(\x)=0$ and use the convention $0\cdot \ln \frac{0}{0}=0$. 
\end{Def}

\begin{Ex} By continuing examples \ref{ex: homomorphisms} and \ref{ex: subgraph counting} we can consider the Bethe partition functions of $\mathcal{G}$ and $\mathcal{G}'$. 
For $\mathcal{G}$ we will denote it by $Z_G^B(N,\mu)$. 
\end{Ex}

Recall that a function $g$ is log-supermodular if for all $\x,\y\in \{0,1\}^k$ we have
$$g(\x)g(\y)\leq g(\x \wedge \y)g(\x \vee \y),$$
where $\x \wedge \y,\x \vee \y \in \{0,1\}^k$ such that
$(\x \wedge \y)_i=\min(x_i,y_i)$ and $(\x \vee \y)_i=\max(x_i,y_i)$ for $i\in [k]$.

\begin{Th}[Ruozzi \cite{ruozzi2012bethe}] Let $\mathcal{G}=(F,V,E,\mathcal{X},(g_a)_{a\in F})$ be a factor graph with $\mathcal{X}=\{0,1\}$ such that
for all $a\in F$ the functions $g_a$ are log-supermodular. Then $Z(\mathcal{G})\geq Z_B(\mathcal{G})$.
\end{Th}

\begin{Lemma}
For an $r \times r$ matrix $N$ and $\vecmu \in \mathbb{R}^r_{\geq 0}$, and for a  $d$-regular graph $G$ we have $Z_G^B(N,\vecmu)\geq \Phi_d(N,\vecmu)^{v(G)}$.
\end{Lemma}

\begin{proof}
By using the same probability distribution $B_{N,\vecmu}$ everywhere in the definition of $Z_G^B(N,\vecmu)$ the consistency of marginals is immediately satisfied. Then the function $\mathbb{F}$ on this pseudo-marginal simplifies to $\mathbb{F}_{\hom}(B_{N,\vecmu})$, and we get that $Z_G^B(N,\vecmu)\geq \Phi_d(N,\vecmu)^{v(G)}$.
\end{proof}

\begin{Th} \label{lower-bound2}
For a $2\times 2$ positive definite matrix $N$ with positive entries and $\vecmu \in \mathbb{R}^2_{\geq 0}$, and a $d$-regular graph $G$ we have
$Z_G(N,\vecmu)\geq \Phi_d(N,\vecmu)^{v(G)}$.
\end{Th}

\begin{proof}
Note that the log-supermodularity of $g_a$ in the case of the factor graph $\mathcal{G}$ in Example~\ref{ex: homomorphisms} simply means that $N_{11}N_{22}\geq N_{12}N_{21}$ which is satisfied as $N$ is positive definite. Hence by combining Ruozzi's theorem with the previous lemma we get that
$$Z_G(N,\vecmu)\geq Z_G^B(N,\vecmu)\geq \Phi_d(N,\vecmu)^{v(G)}.$$
\end{proof}

\begin{Rem}
For Theorem~\ref{lower-bound2} we will give a new proof in Section~\ref{section:v(t0)} that implies a slightly stronger statement. Namely, if $G$ is a  $d$-regular graph such that for some $g$ and $\varepsilon>0$ the graph $G$ contains at least $\varepsilon v(G)$ cycles of length $g$, then $Z_G(N,\vecmu)>((1+\delta)\Phi_d(N,\vecmu))^{v(G)}$ for some $\delta=\delta(d,N,\vecmu,g,\varepsilon)>0$.  
\end{Rem}

\subsection{Convergence of $Z_{G_n}(N,\vecmu)$} \label{sec: convergence}

In this section we prove that if $N$ is a $2\times 2$ positive definite matrix with positive entries and $\vecmu$ is a positive vector, then $\frac{1}{v(G_n)}\ln Z_{G_n}(N,\vecmu)$ is convergent for an essentially large girth sequence of $d$-regular graphs $(G_n)_n$. In fact, we will prove a stronger statement. Namely, we prove the convergence of the sequence $\frac{1}{v(G_n)}\ln Z_{G_n}(N,\vecmu)$ for any Benjamini--Schramm convergent graph sequence $(G_n)_n$ of regular graphs.

\begin{Def} \label{def: BS-convergence}
For a finite graph $G$, a finite connected rooted graph $\alpha$ and a positive integer
$r$, let $\mathbb{P}(G,\alpha,r)$ be the probability that the $r$-ball
centered at a uniform random vertex of $G$ is isomorphic to $\alpha$. 

We say that a bounded-degree graph sequence $(G_n)_n$ is \emph{Benjamini--Schramm
convergent} if for all finite rooted graphs $\alpha$ and $r>0$, the
probabilities $\mathbb{P}(G_n,\alpha,r)$ converge.
\end{Def}

 Benjamini--Schramm convergence is also called \emph{local convergence} as it primarily grasps the local structure of the graphs $(G_n)_n$.
 \medskip

Given a vector $\underline{a}\in \mathbb{R}^{d+1}$ and a $d$-regular graph $G$ on $n$ vertices let $\lambda_1(G),\dots ,\lambda_{nd}(G)$ be the zeros of the polynomial $F_G(\veca|z)$. Let us define the probability measure $\rho_{G,\underline{a}}$ on $\mathbb{C}$ as follows:
$$\rho_{G,\underline{a}}:=\frac{1}{nd}\sum_{k=1}^{nd}\delta_{\lambda_i(G)},$$
where $\delta_{\lambda}$ is the Dirac-measure on the number $\lambda$. 

\begin{Lemma} (a) For any integer $k\geq 0$, a vector $\underline{a}\in \mathbb{R}^{d+1}$ and a Benjamini--Schramm convergent sequence of $d$-regular graphs $(G_n)_n$ the sequence
$$\int z^k\ d\rho_{G_n,\underline{a}}(z)$$
is convergent.

\noindent (b) Let $\vecv_c\in \mathbb{R}^{d+1}$ be such that the zeros of $F_G(\vecv_c|z)$ lie on a circle of radius $R_c$ for all graph $G$. 
If $(G_n)_n$ is a Benjamini--Schramm convergent sequence of $d$-regular graphs, then the sequence of measures $\rho_{G_n,\vecv_c}$ converges weakly. 
\end{Lemma}

\begin{proof}
Part (a) is a special case of a much more general theorem claiming that
$$\int z^k\ d\rho_{G,\underline{a}}(z)=\frac{1}{dv(G)}\sum_{j=1}^{dv(G)}\lambda_j(G)^k$$
can be expressed as $\frac{1}{v(G)}\sum_Hc_{H,k}\hom(H,G)$ for a fixed finite set of \textbf{connected} graphs $H$, and the fact that a sequence of bounded degree graphs $(G_n)_n$ is Benjamini--Schramm convergent if and only if for all connected graphs $H$ the sequence $\frac{\hom(H,G_n)}{v(G_n)}$ is convergent. For details see the paper of Csikv\'ari and Frenkel \cite{csikvari2016benjamini}.

Part (a) implies part (b) for the following reasons. The weak convergence of measures $\rho_n$ on $\mathbb{C}$ is equivalent with the convergence of $\int z^k\overline{z}^{\ell}d\rho_n(z)$ for all integers $k,\ell\geq 0$. But if $\rho_n$ are supported on a fixed circle and they are symmetric to the real line, then this is  equivalent with the convergence of $\int z^md\rho_n(z)$ for all positive integer $m$.
\end{proof}

\begin{Th} \label{convergence_theorem} For any Benjamini--Schramm convergent sequence of $d$-regular graphs $(G_n)_n$ the sequence
$$\frac{1}{v(G_n)}\ln Z_{G_n}(N,\vecmu)$$
is convergent.
\end{Th}

\begin{proof} By Theorem~\ref{zeros-on-circle} there exists a $\vecv_c\in \mathbb{R}^{d+1}$ such that for any $d$-regular graph $G$ we have  $Z_G(N,\vecmu)=F_G(\vecv_c)$ and all zeros of $F_G(\vecv_c|z)$ lie on a circle of radius $R_c(N,\vecmu)$.
First suppose that $R_c=R_c(N,\vecmu)\neq 1$.
We have
\begin{align*}
\frac{1}{v(G)}\ln Z_{G}(N,\mu)&=\frac{1}{v(G)}\ln F_G(\vecv_c|z)\bigg|_{z=1}\\
&=\frac{1}{v(G)}\ln \left(\prod_{j=1}^{dv(G)}(1-\lambda_j(G))\right)\\
&=d\int \ln|z-1|d\rho_{G,\vecv_c}(z).
\end{align*}
The measures $\rho_{G_n,\vecv_c}$ are supported on a circle of radius $R_c\neq 1$, thus $\ln|z-1|$ is a continuous function on a region containing the circle but avoid an open neighborhood of $z=1$. Since the measures $\rho_{G_n,\vecv_c}$ are weakly convergent we get that the sequence $\frac{1}{v(G_n)}\ln Z_{G_n}(N,\vecmu)$ is convergent. 

Next we show that the limit exists even if $R_c(N,\vecmu)=1$. Let $\Phi_L(N,\vecmu)$ be defined by
$$\lim_{n\to \infty}\frac{1}{v(G_n)}\ln Z_{G_n}(N,\vecmu)=\ln \Phi_L(N,\mu)$$
if $R_c(N,\vecmu)\neq 1$. Here $L$ in $\Phi_L(N,\vecmu)$ simply stands for the word limit. We show that $\Phi_L(N,\vecmu)$ is a monotone increasing continuous function of $\mu_1$. Indeed, if $\mu'_1<\mu_1$, then
$$Z_G(N,(\mu'_1,\mu_2))\leq Z_G(N,(\mu_1,\mu_2))\leq \left(\frac{\mu_1}{\mu'_1}\right)^{v(G)}Z_G(N,(\mu'_1,\mu_2)).$$
So if $R_c(N,(\mu'_1,\mu_2))\neq 1$, then
$$\ln \Phi_L(N,(\mu'_1,\mu_2))=\lim_{n\to \infty}\frac{1}{v(G_n)}\ln Z_{G_n}(N,(\mu'_1,\mu_2))\leq 
\liminf_{n\to \infty}\frac{1}{v(G_n)}\ln Z_{G_n}(N,(\mu_1,\mu_2))$$
$$\leq \limsup_{n\to \infty}\frac{1}{v(G_n)}\ln Z_{G_n}(N,(\mu_1,\mu_2))\leq \ln\left(\frac{\mu_1}{\mu'_1}\right)+\ln \Phi_L(N,(\mu'_1,\mu_2)).$$
Note that if $R_c(N,(\mu_1,\mu_2))=1$, then
$$
 2(N_{11}N_{22} - N_{12}^2)+ ( -N_{22}^2T  +  2N_{12}^2  - N_{11}^2T^{-1} )=0,
$$
where $T=\left(\frac{\mu_2}{\mu_1}\right)^{2/d}$.
For fixed $N$ and $\mu_2$ there are at most two $\mu_1$ such that this equation is satisfied. Thus for such a $\mu_1$ we can define
$$\Phi_L(N,(\mu_1,\mu_2))=\lim_{\mu'_1\to \mu_1}\Phi_L(N,(\mu'_1,\mu_2)),$$
and we get that
$$\lim_{n\to \infty}\frac{1}{v(G_n)}\ln Z_{G_n}(N,(\mu_1,\mu_2))=\Phi_L(N,(\mu_1,\mu_2)).$$
\end{proof}

Next we need some result about the number of short cycles in random regular graphs. There are many such results in the literature, we chose one.

\begin{Lemma}[McKay, Wormald, Wysocka \cite{mckay2004short}] \label{cycles-random-regular}
Let $\{c_1,\dots ,c_t\}$ be a non-empty subset of $\{3,\dots ,g\}$. For a random regular graph $G$ of order $n$ and degree $d$, define $M_C(G)=(m_1,\dots ,m_t)$, where $m_i$ is the number of cycles of length $c_i$ in $G$ for $1\leq i\leq t$. For $1\leq i\leq t$ let $\mu_i=\frac{(d-1)^{c_i}}{2c_i}$. Let $S$ be a set of non-negative integer $t$-tuples. Then as $n\to \infty$ the probability that $M_C(G)\in S$ is equal to
$$(1+o(1))\left(\sum_{(m_1,\dots ,m_t)\in S}\prod_{i=1}^t\frac{e^{-\mu_i}\mu_i^{m_i}}{m_i!}\right)+o(1).$$
\end{Lemma}

Now we are ready to give a new proof of the fact that the limit of $\lim_{n\to \infty}\frac{1}{v(G_n)}\ln Z_{G_n}(N,\vecmu)$ is $\ln \Phi_{d}(N,\vecmu)$ for random regular graphs and essentially large girth sequence of regular graphs.

\begin{Th}[Sly and Sun \cite{sly2012computational,sly2014counting} building on Dembo and Montanari \cite{dembo2010ising}] Let $N$ be a $2\times 2$ positive definite matrix with positive entries and let $\vecmu\in \mathbb{R}_{>0}^2$. If $(G_n)_n$ is an essentially large girth sequence of $d$-regular graphs, then
$$\lim_{n\to \infty}\frac{1}{v(G_n)}\ln Z_{G_n}(N,\vecmu)=\ln \Phi_{d}(N,\vecmu).$$
The same statement holds true for a sequence of random $d$-regular graphs with probability one.
\end{Th}

\begin{proof}
 We know from Theorem~\ref{convergence_theorem} that $\lim_{n\to \infty}\frac{1}{v(G_n)}\ln Z_{G_n}(N,\vecmu)$ exists for an essentially large girth sequence of $d$-regular graphs. We only have to prove that this limit is $\ln \Phi_{d}(N,\vecmu)$. To prove this it is enough to show one essentially large girth sequence of $d$-regular graphs $(G_n)_n$ 
for which $\lim_{n\to \infty}\frac{1}{v(G_n)}\ln Z_{G_n}(N,\vecmu)=\ln \Phi_d(N,\vecmu)$. 
Let $G_n$ be a random $d$-regular graph on $n$ vertices, then by Markov's inequality 
$$\mathbb{P}\left(Z_{G_n}(N,\vecmu)\geq n^2\mathbb{E}Z_{G_n}(N,\vecmu)\right)\leq \frac{1}{n^2}.$$
Note that $\mathbb{E}Z_{G_n}(N,\vecmu)=n^C\Phi_{d}(N,\vecmu)^n$ by Theorem~\ref{expected_value}, and for all graph $G$ we have $Z_{G_n}(N,\vecmu)\geq \Phi_{d}(N,\vecmu)^{v(G_n)}$ by Theorem~\ref{lower-bound2}. By Borel-Cantelli lemma we immediately get that $\lim_{n\to \infty}\frac{1}{v(G_n)}\ln Z_{G_n}(N,\vecmu)=\ln \Phi_d(N,\vecmu)$ holds true with probability one.
By Lemma~\ref{cycles-random-regular} we can easily find an essentially large girth sequence of $d$-regular graphs $(G_n)_n$ such that 
$\lim_{n\to \infty}\frac{1}{v(G_n)}\ln Z_{G_n}(N,\vecmu)=\ln \Phi_{d}(N,\vecmu)$.
\end{proof}

\subsection{Trigonometric Bethe approximation} \label{sec: trig-Bethe}
In this subsection we will define some trigonometric polynomial and prove that its maximum is exactly the Bethe approximation.

\begin{Def} By using the notations of Definition~\ref{def_ab} for $\veca,\vecb,\vecmu\in \mathbb{R}^2$ let
$$\phiab(t)=\mu_1a_1(t)^d+\mu_2a_2(t)^d=\mu_1(a_1\cos(t)+b_1\sin(t))^d+\mu_2(a_2\cos(t)+b_2\sin(t))^d.$$
\end{Def}

The following lemma is an immediate consequence of Lemma~\ref{equivalence}.

\begin{Lemma} \label{equivalence-trigonometric-functions}
If $N=\veca\veca^T+\vecb\vecb^T=\hat{\veca}\hat{\veca}^T+\hat{\vecb}\hat{\vecb}^T$, then there exist an $s\in \{-1,1\}$ and an $\alpha\in [0,2\pi]$ such that
$$\Phi_{\hat{\veca},\hat{\vecb},\vecmu}(t)=\phiab(st+\alpha).$$
\end{Lemma}

By Lemma~\ref{equivalence-trigonometric-functions} we can introduce the following concept.

\begin{Def}
Let $N$ be a $2\times 2$ positive definite matrix with positive entries, and let $\vecmu\in \mathbb{R}^2$ be a vector with positive entries. Let $N=\veca\veca^T+\vecb\vecb^T$ be any representation of $N$, then let us define
$$\widetilde{\Phi}_d(N,\vecmu)=\max_{t\in[0,2\pi]}\phiab(t).$$

\end{Def}

The main theorem  of this section is the following.

\begin{Th} \label{th: trigonometric-bethe}
Let $N$ be a $2\times 2$ positive definite matrix with positive entries, and let $\vecmu\in \mathbb{R}^2$ be a vector with positive entries. Then 
$$\Phi_d(N,\vecmu)=\widetilde{\Phi}_d(N,\vecmu).$$
\end{Th}

As a preparation for the proof we introduce some notations.  We will also use the notions and tools from Section~\ref{factor-graphs-Bethe-approximation}. The equation $\mathrm{BP}(h)=h$ using the substitution $R=\frac{h_1}{h_2}$ becomes $$R=\frac{\mu_1}{\mu_2}\left(\frac{N_{11}R+N_{12}}{N_{12}R+N_{22}}\right)^{d-1}. $$
Call $\mathcal{R}_{N,\vecmu}$ the set of non-negative solutions of this equation. By a simple calculation we have
\begin{align*}
    \F(R,N,\vecmu)=&\exp\left(\widetilde{\Phi}_{N,\vecmu,d}(h)\right)=\left(\sum_{\sigma}\mu_{\sigma}\left(\sum_{\sigma'}N_{\sigma,\sigma'}h_{\sigma'}\right)^{d}\right)\left(\sum_{\sigma,\sigma'}N_{\sigma,\sigma'}h_{\sigma}h_{\sigma'}\right)^{-\frac d2}=\\
    &=\mu_1\left[\frac{N_{11}R+N_{12}}{\sqrt{N_{11}R^2+2N_{12}R+N_{22}}}\right]^d+
    \mu_2\left[\frac{N_{12}R+N_{22}}{\sqrt{N_{11}R^2+2N_{12}R+N_{22}}}  \right]^d.
\end{align*}

We know that $$\Phi_d(N,\vecmu)=\max_{R\in \mathcal{R}_{N,\vecmu}}\F(R,N,\vecmu)$$
Let us also choose a representation $N=\veca\veca^T+\vecb\vecb^T$, and let
$$R(t)=-\frac{b_2(t)}{b_1(t)}\ \ \ \text{and}\ \ \  S(t)=\frac{a_1(t)}{a_2(t)}.$$

\begin{Lemma}
Let $a_1,a_2,b_1,b_2,\mu_1,\mu_2\in \mathbb{R}$, then for every $t\in [0,2\pi]$ such that $b_1(t),a_2(t)\neq 0$ we have
$$S(t)=\frac{N_{11}R(t)+N_{12}}{N_{12}R(t)+N_{22}}.$$
Furthermore,
$$|a_1(t)|=\frac{|N_{11}R(t)+N_{12}|}{\sqrt{N_{11}R(t)^2+2N_{12}R(t)+N_{22}}}$$
    and $$|a_2(t)|=\frac{|N_{12}R(t)+N_{22}|}{\sqrt{N_{11}R(t)^2+2N_{12}R(t)+N_{22}}},$$
and if $t_0$  maximizes $\phiab(t)$, then $$R(t_0)=\frac{\mu_1}{\mu_2}\left(\frac{N_{11}R(t_0)+N_{12}}{N_{12}R(t_0)+N_{22}}\right)^{d-1}.$$
\end{Lemma}

%    $$S=\frac{a_1}{a_2}=\frac{\frac{N_{11}R+N_{12}}{\sqrt{N_{11}R^2+2N_{12}R+N_{22}}}}{\frac{N_{12}R+N_{22}}{\sqrt{N_{11}R^2+2N_{12}R+N_{22}}}}.$$

\begin{proof} We have
\begin{align*}
\frac{N_{11}R(t)+N_{12}}{N_{12}R(t)+N_{22}}&=\frac{N_{11}\left(-\frac{b_2(t)}{b_1(t)}\right)+N_{12}}{N_{12}\left(-\frac{b_2(t)}{b_1(t)}\right)+N_{22}}\\
&=\frac{-N_{11}b_2(t)+N_{12}b_1(t)}{-N_{12}b_2(t)+N_{22}b_1(t)}\\
&=\frac{(a_1^2+b_1^2)(a_2\sin(t)-b_2\cos(t))+(a_1a_2+b_1b_2)(-a_1\cos(t)+b_1\sin(t))}{(a_1a_2+b_1b_2)(a_2\sin(t)-b_2\cos(t))+(a_2^2+b_2^2)(-a_1\cos(t)+b_1\sin(t))}\\
&=\frac{b_1(a_2b_1-a_1b_2)\sin(t)+a_1(a_2b_1-a_1b_2)\cos(t)}{b_2(a_2b_1-a_1b_2)\sin(t)+a_2(a_2b_1-a_1b_2)\cos(t)}\\
&=\frac{b_1\sin(t)+a_1\cos(t)}{b_2\sin(t)+a_2\cos(t)}\\
&=\frac{a_1(t)}{a_2(t)}\\
&=S(t)
\end{align*}
Note that  we have $(a_2b_1-a_1b_2)^2=(a_1^2+a_2^2)^2(b_1^2+b_2^2)^2-(a_1a_2+b_1b_2)^2=N_{11}N_{22}-N_{12}^2\neq 0$. Next let us prove that
$$|a_1(t)|=\frac{|N_{11}R(t)+N_{12}|}{\sqrt{N_{11}R(t)^2+2N_{12}R(t)+N_{22}}}.$$
Let us multiply both sides with the denominator of the right hand side, and take the square of both sides. Note that by Lemma~\ref{identities} and the decomposition $N=\veca\veca^T+\vecb\vecb^T$ we have $N_{11}=a_1(t)^2+b_1(t)^2$, $N_{12}=a_1(t)a_2(t)+b_1(t)b_2(t)$ and  $N_{22}=a_2(t)^2+b_2(t)^2$ are true for every $t$.
For ease of notation let 
$$\mathrm{LHS}=a_1(t)^2( N_{11}R(t)^2+2N_{12}R(t)+N_{22}).$$
Then
\begin{align*}
\mathrm{LHS}&=a_1(t)^2( N_{11}R(t)^2+2N_{12}R(t)+N_{22})\\
&=a_1(t)^2\left((a_1(t)^2+b_1(t)^2)\frac{b_2(t)^2}{b_1(t)^2}-2(a_1(t)a_2(t)+b_1(t)b_2(t))\frac{b_2(t)}{b_1(t)}+(a_2(t)^2+b_2(t)^2)\right)\\
&=a_1(t)^4\frac{b_2(t)^2}{b_1(t)^2}-2a_1(t)^3a_2(t)\frac{b_2(t)}{b_1(t)}+a_1(t)^2a_2(t)^2\\
&=\left(-a_1(t)^2\frac{b_2(t)}{b_1(t)}+a_1(t)a_2(t)\right)^2\\
&=\left(-(a_1(t)^2+b_1(t)^2)\frac{b_2(t)}{b_1(t)}+(a_1(t)a_2(t)+b_1(t)b_2(t))\right)^2\\
&=\left(N_{11}R(t)+N_{22}\right)^2
\end{align*}
The proof of the third identity follows similarly and we omit it. 

If $t_0$ maximizes $\phiab(t)$, in fact, we only need $\phiab'(t_0)=0$, then we have $$R(t_0)=\frac{\mu_1}{\mu_2}S(t_0)^{d-1}=\frac{\mu_1}{\mu_2}\left(\frac{N_{11}R(t_0)+N_{12}}{N_{12}R(t_0)+N_{22}}\right)^{d-1}.$$

\end{proof}

Now we are ready to prove Theorem~\ref{th: trigonometric-bethe}.

\begin{proof}[Proof of Theorem~\ref{th: trigonometric-bethe}]
Note that $\widetilde{\Phi}_d(N,\vecmu)$ does not depend on which  representation \\ 
$N=\veca\veca^T+\vecb\vecb^T$ we choose so we can assume by part (ii) of Lemma~\ref{lemma:decompositions} that $a_1,a_2,b_1,b_2>0$. Then 
$$\phiab(t)=\mu_1(a_1\cos(t)+b_1\sin(t))^d+\mu_2(a_2\cos(t)+b_2\sin(t))^d$$
is maximized at some $t_0\in \left[0,\frac{\pi}{2}\right]$. Then $S(t_0)>0$, and by $R(t_0)=\frac{\mu_1}{\mu_2}S(t_0)^{d-1}$ we get that $R(t_0)>0$. Thus $R(t_0)\in \mathcal{R}_{N,\mu}$ and we have 
$$\widetilde{\Phi}_d(N,\vecmu)=\phiab(t_0)=\F(R(t_0),N,\vecmu)\le\max_{R\in \mathcal{R}_{N,\vecmu}}\F(R,N,\vecmu)=\Phi_d(N,\vecmu). $$
On the other hand,
$\Phi_d(N,\vecmu)=\max_{R\in \mathcal{R}_{N,\vecmu}}\F(R,N,\vecmu)=\F(R_0,N,\vecmu)$ for some $R_0$. Then by Lemma~\ref{all-value} there exists a $t'\in[0,2\pi]$ such that $R(t')=R_0$. Note that from
$$\frac{a_1(t')}{a_2(t')}=S(t')=\frac{N_{11}R_0+N_{12}}{N_{12}R_0+N_{22}}>0,$$
we get that $a_1(t'),a_2(t')$ have the same sign. By changing $t'$ to $t'+\pi$ if necessary we can ensure that they are both positive. Hence we have 
$$ a_1(t')=\frac{N_{11}R_0+N_{12}}{\sqrt{N_{11}R_0^2+2N_{12}R_0+N_{22}}}\ \text{and}\ \  a_2(t')=\frac{N_{12}R_0+N_{22}}{\sqrt{N_{11}R_0^2+2N_{12}R_0+N_{22}}}.$$ 
Then
$$\Phi_d(N,\vecmu)=\F(R_0,N,\vecmu)=\phiab(t')\leq \max_{t\in [0,2\pi]}\phiab(t)=\widetilde{\Phi}_d(N,\vecmu).$$

\end{proof}

We end this section with a lemma that we will use later.

\begin{Lemma} \label{number-solutions}
Let $Q$ be a $2\times 2$ matrix with positive entries and positive determinant, and let $k$ be an integer. Then the equation 
$$R=\left(\frac{Q_{11}R+Q_{12}}{Q_{21}R+Q_{22}}\right)^{k}$$
has at most $3$ non-negative solutions.
\end{Lemma}

\begin{proof}
Let 
$$f(R)=\left(\frac{Q_{11}R+Q_{12}}{Q_{21}R+Q_{22}}\right)^{k}-R.$$
Then $\frac{\partial^2}{\partial R^2}f$ is given as
$$\left(\frac{Q_{11}R+Q_{12}}{Q_{21}R+Q_{22}}\right)^{k}\frac{k(Q_{11}Q_{22}-Q_{12}Q_{21})((k-1)Q_{11}Q_{22}-(k+1)Q_{12}Q_{21}-2Q_{11}Q_{21}R)}{(Q_{11}R+Q_{12})^2(Q_{12}R+Q_{22})^2}.$$
Let 
$$R^*=\frac{(k-1)Q_{11}Q_{22}-(k+1)Q_{12}Q_{21}}{2Q_{11}Q_{21}}.$$ 
If $R^*\leq 0$, then $f$ is concave of $[0,\infty)$ and so it has at most two solutions.

If $R^*> 0$, then  $\frac{\partial^2}{\partial R^2}f$ is positive if $0\leq R<R^*$, and negative if $R>R^*$. If $f(R^*)>0$, then $f$ has at most $2$ solutions on $(0,R^*)$ and has at most $1$ solution on $(R^*,\infty)$.
 If $f(R^*)<0$, then $f$ has at most $1$ solution on $(0,R^*)$ and has at most $2$ solutions on $(R^*,\infty)$. Finally, if $f(R^*)=0$, then $f$ has at most $1$ solution on $(0,R^*)$ and has at most $1$ solutions on $(R^*,\infty)$. So in all cases it has at most $3$ solutions.

\end{proof}

\subsection{The vector $\vecv(t_0)$} \label{section:v(t0)}
Let $t_0$ be the maximizer of the  function $\phiab$. 
In this section we study the vector $\vecv(t_0)=(r_0(t_0),r_1(t_0),\dots ,r_d(t_0))$. 
First we need a simple lemma.

\begin{Lemma} \label{vector-0-I}
Let $\veca,\vecb,\vecmu\in \mathbb{R}^r$. For $j=1,\dots ,r$ let
$$a_j(t)=a_j\cos(t)+b_j\sin(t)\ \ \ \text{and}\ \ \ 
b_j(t)=-a_j\sin(t)+b_j\cos(t).$$
Let $r_j(t)=\sum_{k=1}^r\mu_ka_k(t)^{d-j}b_k(t)^{j}$, then
$$r_0(t)=\sum_{j=0}^d\binom{d}{j}r_j(0)\cos(t)^{d-j}\sin(t)^j.$$
\end{Lemma}

\begin{proof}
We have
\begin{align*}
r_0(t)&=\sum_{k=1}^r\mu_ka_k(t)^d\\
&=\sum_{k=1}^r\mu_k(a_k\cos(t)+b_k\sin(t))^d\\
&=\sum_{k=1}^r\mu_k\sum_{j=0}^d\binom{d}{j}a_k^{d-j}b_k^{j}\cos(t)^{d-j}\sin(t)^j\\
&=\sum_{j=0}^d\binom{d}{j}\cos(t)^{d-j}\sin(t)^j\left(\sum_{k=1}^r\mu_ka_k^{d-j}b_k^{j}\right)\\
&=\sum_{j=0}^d\binom{d}{j}r_j(0)\cos(t)^{d-j}\sin(t)^j
\end{align*}

\end{proof}

Consider the vector
$\vecv(t_0)=(r_0(t_0),r_1(t_0),\dots ,r_d(t_0))$. We show that $r_1(t_0)=0$ and $r_j(t_0)\geq 0$ if $j$ is even, and the numbers $r_j(t_0)$ have the same sign for odd $j\geq 3$. This follows from the following more general lemma.

\begin{Lemma} \label{vector-t0}
Let $\mu_1,\mu_2>0$ and $a_1,a_2,b_1,b_2\in \mathbb{R}$ such that $a_1,a_2,b_1,b_2>0$. Let $t_0\in \left[0,\frac{\pi}{2}\right]$ be the maximizer of
$\Phi_{\veca,\vecb,\vecmu}(t)=\mu_1a_1(t)^d+\mu_2a_2(t)^d$.
Let 
$$r_j(t)=\mu_1a_1(t)^{d-j}b_1(t)^j+\mu_2a_2(t)^{d-j}b_2(t)^j.$$
Then $r_1(t_0)=0$ and either\\
(i) $r_j(t_0)\geq 0$ for $j=0,\dots ,d$ or\\
(ii) $r_j(t_0)\geq 0$ for even $j$, and $r_j(t_0)\leq 0$ for odd $j$.\\
\end{Lemma}

\begin{proof}
Observe that $\frac{\partial}{\partial t}a_1(t)=b_1(t)$ and $\frac{\partial}{\partial t}a_2(t)=b_2(t)$ and so
$$\frac{\partial}{\partial t}r_0(t)=d\mu_1a_1(t)^{d-1}b_1(t)+d\mu_2a_2(t)^{d-1}b_2(t)=dr_1(t).$$
Hence if $t_0$ maximizes $r_0(t)$, then $r_1(t_0)=0$.

To prove the inequalities we need to study several cases. First of all, $r_1(t_0)=0$ implies that $a_1(t_0)b_1(t_0)=0$ if and only if $a_2(t_0)b_2(t_0)=0$. If $a_1(t_0)b_1(t_0)=a_2(t_0)b_2(t_0)=0$, then $r_1(t_0)=r_2(t_0)=\dots =r_{d-1}(t_0)=0$. We know that 
$r_0(t_0)\geq r_0(0)=\mu_1a_1^d+\mu_2a_2^d>0$. Finally,
$r_d(t_0)=\mu_1b_1(t_0)^d+\mu_2b_1(t_0)^d\geq 0$ if $d$ is even. If $d$ is odd, then no matter what the sign of $r_d(t_0)$ case (i) or (ii) is satisfied.

So we can assume that $a_1(t_0)b_1(t_0)\neq 0$ and  $a_2(t_0)b_2(t_0)\neq 0$. By symmetry we can assume that 
$|\mu_2a_2(t_0)^d|\geq |\mu_1a_1(t_0)^d|$. Note that
$$r_j(t_0)=\mu_1a_1(t_0)^d\left(\frac{b_1(t_0)}{a_1(t_0)}\right)^j\left(1+\frac{\mu_2a_2(t_0)^d}{\mu_1a_1(t_0)^d}\left(\frac{a_1(t_0)b_2(t_0)}{a_2(t_0)b_1(t_0)}\right)^j\right).$$
From $r_1(t_0)=0$ we get that 
$$\left|\frac{a_1(t_0)b_2(t_0)}{a_2(t_0)b_1(t_0)}\right|=\left|\frac{\mu_1a_1(t_0)^d}{\mu_2a_2(t_0)^d}\right|\leq 1.$$

Note that if $\mu_2a_2(t_0)^d=\mu_1a_1(t_0)^d>0$, then  $\frac{a_1(t_0)b_2(t_0)}{a_2(t_0)b_1(t_0)}=-1$. Then for odd $j$ we have $r_j(t_0)=0$, and for even $j$ all terms are positive in the above product, so $r_j(t_0)>0$. 

If $\mu_2a_2^d>\mu_1a_1^d$, then $\left|\frac{a_1(t_0)b_2(t_0)}{a_2(t_0)b_1(t_0)}\right|<1$, and so the last term is positive for all $j\geq 2$. Then if $\frac{b_1(t_0)}{a_1(t_0)}> 0$, then $r_j(t_0)\geq 0$ for all $j$, and if $\frac{b_1(t_0)}{a_1(t_0)}<0$, then $r_j(t)$ is positive for even $j$ and negative for odd $j\geq 3$. We are done.
\end{proof}

\begin{Rem} \label{second-derivative bound}
Note that besides $r_1(t_0)=0$ we also have $\frac{r_2(t_0)}{r_0(t_0)}\leq \frac{1}{d-1}$ since 
$\frac{\partial^2}{\partial t^2}r_0(t)=d((d-1)r_2(t)-r_0(t))$ should be non-positive at $t=t_0$.
\end{Rem}

The following theorem is a strengthening of Theorem~\ref{lower-bound2} with a  new proof. 

\begin{Th} \label{lower-bound}
Let $N$ be a $2\times 2$ positive definite matrix with positive entries and let $\vecmu\in \mathbb{R}_{>0}^2$. For any $d$-regular graph $G$ we have $Z_G(N,\vecmu)\geq \Phi_{d}(N,\vecmu)^{v(G)}$. Furthermore, if $G$ contains $\varepsilon v(G)$ cycles of length at most $g$, then there exists a $\delta=\delta(d,N,\vecmu,\varepsilon,g)>0$ such that
$Z_G(N,\vecmu)\geq ((1+\delta)\Phi_{d}(N,\vecmu))^{v(G)}$.
\end{Th}

The proof presented below is strongly inspired by the work of Chertkov and Chernyak \cite{chertkov2006loop1} on loop series and gauge transformation. The paper of Borb\'enyi and Csikv\'ari \cite{borbenyi2020counting} contains a similar proof about the number of Eulerian orientations in regular graphs.

\begin{proof}
By part (ii) of Lemma~\ref{lemma:decompositions} we can choose $\veca,\vecb\in \mathbb{R}^2$ such that $a_1,a_2,b_1,b_2>0$. Then $r_j(0)=\mu_1a_1^{d-j}b_1^j+\mu_2a_2^{d-j}b_2^j>0$ for all $j\in \{0,1,\dots ,d\}$. This implies that
$$\phiab(t)=r_0(t)=\sum_{j=0}^d\binom{d}{j}r_j(0)\cos(t)^{d-j}\sin(t)^j$$
has a maximizer $t_0$ in the interval $\left[0,\frac{\pi}{2}\right]$. Indeed, for any $t\in [0,2\pi]$ there is a $t'\in \left[0,\frac{\pi}{2}\right]$ such that $\cos(t')=|\cos(t)|$ and $\sin(t')=|\sin(t)|$, so $|\phiab(t)|\leq \phiab(t')$. Thus the conditions of Lemma~\ref{vector-t0} are satisfied.
We have
$$Z_G(N,\vecmu)=F_G(r_0(t_0),\dots ,r_d(t_0))=
\sum_{A \subseteq E(G)}\prod_{v\in V}r_{d_A(v)}(t_0).$$
For each $A\subseteq E(G)$ the number of vertices with odd $d_A(v)$ is even, so by Lemma~\ref{vector-t0} each term in the sum is non-negative. Then taking $A=\emptyset$ we get that
$$Z_G(N,\vecmu)\geq r_0(t_0)^{v(G)}=\Phi_{d}(N,\vecmu)^{v(G)}.$$
This completes the proof of the first part.

To prove the second part first observe that
$r_2(t_0)>0$. Indeed,  since $a_1,a_2,b_1,b_2>0$ and $t_0\in \left[0,\frac{\pi}{2}\right]$ we get that $a_1(t_0),a_2(t_0)>0$ and thus $$r_2(t_0)=\mu_1a_1(t_0)^{d-2}b_1(t_0)^2+\mu_2a_2(t_0)^{d-2}b_2(t_0)^2=0$$
would imply that $b_1(t_0)=b_2(t_0)=0$ which then implies that
$$N_{11}N_{22}-N_{12}^2=(a_1(t_0)b_2(t_0)-a_2(t_0)b_1(t_0))^2=0$$
contradicting the positive definiteness of $N$.
Also observe that if $G$ contains $\varepsilon v(G)$ cycles of length at most $g$, then it also contains $\varepsilon' v(G)$ vertex-disjoint cycles of length at most $g$ for some $\varepsilon'$ depending on $d,g$ and $\varepsilon$, but not depending on $v(G)$.
Then 
$$Z_G(N,\vecmu)\geq \Phi_d(N,\vecmu)^{v(G)}\left(1+\left(\frac{r_2(t_0)}{r_0(t_0)}\right)^g\right)^{\varepsilon' v(G)}.$$
Indeed, we can consider those sets $A\subseteq E(G)$ that consists of the union of some of the vertex-disjoint cycles of length at most $g$. Here we also use the fact that $0\leq \frac{r_2(t_0)}{r_0(t_0)}\leq \frac{1}{d-1}\leq 1$ by Remark~\ref{second-derivative bound}. Hence $1+\delta=\left(1+\left(\frac{r_2(t_0)}{r_0(t_0)}\right)^g\right)^{\varepsilon'}$ satisfies the claim of the theorem.
\end{proof}

\begin{Rem}
Theorem~\ref{lower-bound} implies that if $(G_n)_n$ is a sequence of $d$-regular graphs such that it is not essentially large girth, then
$$\limsup_{n\to \infty}\frac{1}{v(G_n)}\ln Z_{G_n}(N,\vecmu)>\ln \Phi_d(N,\vecmu).$$
Since $Z_G(q,w)\geq Z_{G}^{(2)}(q,w)$ for $q\geq 2$ and $w\geq 0$ this kind of stability statement is also true for $Z_G(q,w)$.
\end{Rem}

\subsection{Mixed state}
In this section we introduce a concept that is strongly related to the phase transition of the random cluster model. We will see that there exists a $w_c=w_c(q)$ such that if $0\leq w\leq w_c$, then $\Phi_{d,q,w}=q\left(1+\frac{w}{q}\right)^{d/2}$, and if $w>w_c$, then $\Phi_{d,q,w}>q\left(1+\frac{w}{q}\right)^{d/2}$. The problem with such a statement is that it depends on the parametrization $(q,w)$, for a general model $(N,\vecmu)$ it does not make sense. On the other hand, there is a concept that makes sense even for general $(N,\vecmu)$, where $N$ is a $2\times 2$ positive definite matrix.

\begin{Def}
We say that $(N,\vecmu)$ exhibits a mixed state for a fixed positive integer $d$ if $R_c(N,\vecmu)=1$.
\end{Def}

Note that $R_c(N,\vecmu)=1$ does not depend on which representation $N=\veca\veca^T+\vecb\vecb^T$ we choose. 
We also know that $R=R_c(N,\vecmu)$ is a solution of 
$$
 (N_{11}N_{22} - N_{12}^2) R^4 + ( -N_{22}^2T  +  2N_{12}^2  - N_{11}^2T^{-1} ) R^2 + (N_{11}N_{22}- N_{12}^2)=0,
$$
where $T=\left(\frac{\mu_2}{\mu_1}\right)^{2/d}$. This shows that $(N,\vecmu)$ exhibits a mixed state for $d$ if 
$$2(N_{11}N_{22} - N_{12}^2)-(N_{22}^2T  -  2N_{12}^2  + N_{11}^2T^{-1} )=0.$$

The main lemma of this section is the following. 

\begin{Lemma} \label{symmetry-mixed-state}
Let $N=\veca\veca^T+\vecb\vecb^T$ for some $a_1,a_2,b_1,b_2\in \mathbb{R}$ and let $\mu_1,\mu_2>0$.  Suppose that  $(N,\vecmu)$ exhibits a mixed state for $d$, that is, for some $t_1$ we have
$$a_1(t_1)=\left(\frac{\mu_2}{\mu_1}\right)^{1/d}(-b_2(t_1))\ \ \text{and}\ \ \ -b_1(t_1)=\left(\frac{\mu_2}{\mu_1}\right)^{1/d}a_2(t_1).$$
Let $t_2=2t_1-\frac{\pi}{2}$. Then for every $t\in \mathbb{R}$ we have
$$\mu_1a_1(t)^d=\mu_2a_2(t_2-t)^d\ \ \ \text{and}\ \ \ \mu_2a_2(t)^d=\mu_1a_1(t_2-t)^d.$$
In particular,
$$\phiab(t)=\phiab(t_2-t).$$
\end{Lemma}

\begin{proof}
Note that for any $u_1,u_2\in \mathbb{R}$ we have
$$a_1(u_1+u_2)=a_1(u_1)\cos(u_2)+b_1(u_1)\sin(u_2).$$
To prove $\mu_1a_1(t)^d=\mu_2a_2(t_2-t)^d$ it will be more convenient to prove the statement in the form
$$a_2(t_1-t)=\left(\frac{\mu_1}{\mu_2}\right)^{1/d}a_1\left(t+t_1-\frac{\pi}{2}\right).$$
This is indeed true,
\begin{align*}
a_2(t_1-t)&=a_2(t_1)\cos(-t)+b_2(t_1)\sin(-t)\\
&=a_2(t_1)\cos(t)-b_2(t_1)\sin(t)\\
&=\left(\frac{\mu_1}{\mu_2}\right)^{1/d}(-b_1(t_1)\cos(t)+a_1(t_1)\sin(t))\\
&=\left(\frac{\mu_1}{\mu_2}\right)^{1/d}\left(a_1(t_1)\cos\left(t-\frac{\pi}{2}\right)+b_1(t_1)\sin\left(t-\frac{\pi}{2}\right)\right)\\
&=\left(\frac{\mu_1}{\mu_2}\right)^{1/d}a_1\left(t+t_1-\frac{\pi}{2}\right)
\end{align*}
By symmetry the other claim is also true.
\end{proof}

\subsection{Specialization to $N=M'_2$ and $\vecmu=\underline{\nu}_2$}

In this section we collected some results that are specialized to $N=M'_2$ and $\vecmu=\underline{\nu}_2$. In particular, we will choose $a_1=a_2=\sqrt{1+\frac{w}{q}}$, $b_1=\sqrt{\frac{(q-1)w}{q}}$ and $b_2=-\sqrt{\frac{w}{q(q-1)}}$.

\begin{Lemma} \label{vector-0-II}
Let $q\geq 2,w\geq 0$. Let $a_1=a_2=\sqrt{1+\frac{w}{q}}$, $b_1=\sqrt{\frac{(q-1)w}{q}}$, $b_2=-\sqrt{\frac{w}{q(q-1)}}$, $\nu_1=1$ and $\nu_2=q-1$.
Let $r_j(0)=\nu_1a_1^{d-j}b_1^j+\nu_2a_2^{d-j}b_2^j$. 
Then we have $r_j(0)\geq 0$ for $j=0,1,\dots ,d$ and $r_1(0)=0$.
\end{Lemma}

\begin{proof}
We have 
\begin{align*}
r_j(0)&=\left(1+\frac{w}{q}\right)^{(d-j)/2}\left(\frac{(q-1)w}{q}\right)^{j/2}+(q-1)(-1)^j\left(1+\frac{w}{q}\right)^{(d-j)/2}\left(\frac{w}{q(q-1)}\right)^{j/2}\\
&=\left(1+\frac{w}{q}\right)^{(d-j)/2}\left(\frac{w}{q(q-1)}\right)^{j/2}((q-1)^j+(-1)^j(q-1))
\end{align*}
This is $0$ if $j=1$, and positive if $j\neq 1$.
\end{proof}

Recall that
$$\Phi_{d,q,w}=\max_{t\in [0,2\pi]}\Phi_{d,q,w}(t).$$
The next lemma shows that it is enough to consider  the interval  $\left[0,\frac{\pi}{2}\right]$ to find the maximum when $q\geq 2$ and $w\geq 0$.

\begin{Lemma}
If $q\geq 2$ and $w\geq 0$, then there is a $t_0\in \left[0,\frac{\pi}{2}\right]$ for which  $\Phi_{d,q,w}=\Phi_{d,q,w}(t_0)$.
\end{Lemma}

\begin{proof}
By Lemma~\ref{vector-0-I} we have
$$\Phi_{d,q,w}(t)=r_0(t)=\sum_{j=0}^d\binom{d}{j}r_j(0)\cos(t)^{d-j}\sin(t)^j.$$
By Lemma~\ref{vector-0-II} we have $r_j(0)\geq 0$ for all $j\in \{0,1,\dots ,d\}$ if $q\geq 2$ and $w\geq 0$. For any $t\in [0,2\pi]$ there is a $t'\in \left[0,\frac{\pi}{2}\right]$ such that $|\cos(t)|=\cos(t')$ and $|\sin(t)|=\sin(t')$, thus 
$$|\Phi_{d,q,w}(t)|\leq \sum_{j=0}^d\binom{d}{j}r_j(0)|\cos(t)|^{d-j}|\sin(t)|^j= \Phi_{d,q,w}(t').$$
Hence $\max_{t\in [0,2\pi]}\Phi_{d,q,w}(t)=\max_{t\in \left[0,\frac{\pi}{2}\right]}\Phi_{d,q,w}(t).$
\end{proof}

The next lemma will be useful to get even more precise bounds on $\tan(t_0)$.

\begin{Lemma}\label{Lemma:bethesigns}
Let $q\geq 2$ and $w\geq 0$. Let $\ovt\in \left[0,\frac{\pi}{2}\right]$ such that
$$\frac{\partial}{\partial t}\Phi_{d,q,w}(t)\bigg|_{t=\ovt}=0.$$
Then we have $a_{q,w,1}(\ovt),a_{q,w,2}(\ovt),b_{q,w,1}(\ovt)>0$ and $b_{q,w,2}(\ovt)<0$. In particular, this is true if $\ovt=t_0$ maximizing the function $\Phi_{d,q,w}(t)$ in the interval $\left[0,\frac{\pi}{2}\right]$.
\end{Lemma}

\begin{proof}
Note that $\Phi_{d,q,w}(t)=a_{q,w,1}(t)^d+(q-1)a_{q,w,2}(t)^d$, and its derivative is 
$$\frac{\partial}{\partial t}\Phi_{d,q,w}(t)=d(b_{q,w,1}(t)a_{q,w,1}(t)^{d-1}+b_{q,w,2}(t)a_{q,w,2}(t)^{d-1}).$$
Note that $a_{q,w,1}(t)>0$ and $b_{q,w,2}(t)<0$ for all $t\in \left[0,\frac{\pi}{2}\right]$. Suppose for contradiction that $b_{q,w,1}(\ovt)<0$. Then from $b_{q,w,1}(\ovt)a_{q,w,1}(\ovt)^{d-1}+b_{q,w,2}(\ovt)a_{q,w,2}(\ovt)^{d-1}=0$
we also get that $a_{q,w,2}(\ovt)^{d-1}<0$, that is, $a_{q,w,2}(\ovt)<0$ and $d$ is even. Then
$$\left(\frac{a_{q,w,1}(\ovt)}{-a_{q,w,2}(\ovt)}\right)^{d-1}=\frac{(q-1)(-b_{q,w,2}(\ovt))}{-b_1(\ovt)}.$$
Note that $a_{q,w,1}(t)>-a_{q,w,2}(t)$ for all $t\in \left[0,\frac{\pi}{2}\right]$, and so
$$\left(\frac{a_{q,w,1}(\ovt)}{-a_{q,w,2}(\ovt)}\right)^{d-1}>\frac{a_{q,w,1}(\ovt)}{-a_{q,w,2}(\ovt)}.$$
By $a_{q,w,1}(t)b_{q,w,1}(t)+(q-1)a_{q,w,2}(t)b_{q,w,2}(t)=-q\cos(t)\sin(t)$ we have
$$\frac{a_{q,w,1}(\ovt)}{-a_{q,w,2}(\ovt)}>\frac{(q-1)(-b_{q,w,2}(\ovt))}{-b_{q,w,1}(\ovt)},$$
but then
$$\left(\frac{a_{q,w,1}(\ovt)}{-a_{q,w,2}(\ovt)}\right)^{d-1}>\frac{a_{q,w,1}(\ovt)}{-a_{q,w,2}(\ovt)}>\frac{(q-1)(-b_{q,w,2}(\ovt))}{-b_{q,w,1}(\ovt)}$$
leads to a contradiction. 
Hence $b_{q,w,1}(\ovt)>0$. But then $a_{q,w,1}(\ovt)a_{q,w,2}(\ovt)+b_{q,w,1}(\ovt)b_{q,w,2}(\ovt)=1$ implies that 
$a_{q,w,2}(\ovt)>0$.

\end{proof}

Finally, we collected some claims about the derivatives of $\Phi_{d,q,w}(t)$ at $t=0.$

\begin{Lemma}
We have
$$\frac{\partial}{\partial t}\Phi_{d,q,w}(t)\bigg|_{t=0}=0\ \ \ \text{and}\ \  \ \frac{\partial^2}{\partial t^2}\Phi_{d,q,w}(t)\bigg|_{t=0}=\left(1+\frac{w}{q}\right)^{d/2-1}((d-2)w-q).$$
In particular, if $w<\frac{q}{d-2}$, then the function 
$\Phi_{d,q,w}(t)$ has a local maximum at $t=0$, and if $w>\frac{q}{d-2}$ the function $\Phi_{d,q,w}(t)$ has a local minimum at $t=0.$
\end{Lemma}

\subsubsection{Mixed state and phase transition}
In this section we discuss the mixed state and phase transition of $Z^{(2)}_G(q,w)$.

We know that $(N,\vecmu)$ exhibits mixed state for some $d$ if 
$$2(N_{11}N_{22} - N_{12}^2)-(N_{22}^2T  -  2N_{12}^2  + N_{11}^2T^{-1} )=0,$$
where $T=\left(\frac{\mu_2}{\mu_1}\right)^{1/d}$.
Applying this equation to $(M'_2,\underline{\nu}_2)$ we get that
$$2\left((1+w)\left(1+\frac{w}{q-1}\right)-1\right)-\left(\left(1+\frac{w}{q-1}\right)^2(q-1)^{2/d}-2+(1+w)^2(q-1)^{-2/d}\right)=$$
$$-(((q-1)^{-1/d}-(q-1)^{1/d-1})w-((q-1)^{1/d}-(q-1)^{-1/d}))^2.$$
Note that if $q=2$, then this is constant $0$, so in the case of $q=2$, the spin system $(M'_2,\underline{\nu}_2)$ always exhibits mixed state. If $q\neq 2$, then the solution $w_c=w_c(q)$ of this equation is
$$w_c=\frac{(q-1)^{1/d}-(q-1)^{-1/d}}{(q-1)^{-1/d}-(q-1)^{1/d-1}}=\frac{(q-1)-(q-1)^{1-2/d}}{(q-1)^{1-2/d}-1}=\frac{q-2}{(q-1)^{1-2/d}-1}-1.$$
Note that by L'H\^{o}pital's rule we have
$$\lim_{q\to 2+}w_c(q)=\frac{2}{d-2},$$
so we will define $w_c(2)=\frac{2}{d-2}$ even though the 
spin system $(M'_2,\underline{\nu}_2)$ itself always exhibits mixed state for every $w$ if $q=2$.

The main theorem of this section is the following. It asserts that the mixed state also describes a phase transition in the value of $\Phi_{d,q,w}$.

\begin{Th} \label{th: phase-transition} Let $q\geq 2$.
If $0\leq w\leq w_c$, then $\Phi_{d,q,w}=q\left(1+\frac{w}{q}\right)^{d/2}$. If $w>w_c$, then $\Phi_{d,q,w}>q\left(1+\frac{w}{q}\right)^{d/2}$.
\end{Th}

Before we start to prove Theorem~\ref{th: phase-transition} we need a lemma about the curve $(q,w_c(q))$.  For a visualization of this lemma see the dashed curve on Figure~\ref{fig:without_colors}.

\begin{Lemma} \label{qw-xy-reparametrization}
For every $q\geq 2$ let $x(q)=1+\frac{q}{w_c(q)}$ and $y(q)=1+w_c(q)$. Then the curve $(x(q),y(q))$ on the $(x,y)$ plane is the graph of a monotone increasing function. Furthermore, $w_c(q)\leq \frac{q}{d-2}$.
\end{Lemma}

\begin{proof}
We have
$$y=w_c(q)+1=\frac{q-2}{(q-1)^{1-2/d}-1},$$
and $q=(x(q)-1)(y(q)-1)$.
Thus
$$\frac{\partial y}{\partial x}=\frac{\partial y}{\partial q}\cdot \frac{\partial q}{\partial x}=\frac{(q-1)^{1-2/d}-1-(1-2/d)(q-1)^{-2/d}(q-2)}{((q-1)^{1-2/d}-1)^2}\cdot (y-1).$$
Clearly, $y-1=w>0$ so we only need to show that the first term is also positive. We can rewrite the numerator as 
$$\frac{2}{d}(q-1)^{-2/d}(q-2)+(q-1)^{-2/d}-1=(q-1)^{-2/d}\left(\frac{2}{d}(q-2)+1-(q-1)^{2/d}\right).$$
This is $0$ if $q=2$ and its derivative is $\frac{2}{d}
(1-(q-1)^{-2/d})\geq 0$ for $q\geq 2$. Hence $\frac{dy
}{dx}>0$.

The second part follows from the first part since if we follow any hyperbola $(x-1)(y-1)=q$ while decreasing $x$, and hence increasing $y$, it intersects the curve $(x(q),y(q))$ before hitting the line $x=d-1$.  The intersection of the hyperbola 
$(x-1)(y-1)=q$ and the line $x=d-1$ is at $y=1+\frac{q}{d-2}$, thus $w_c(q)\leq \frac{q}{d-2}$.
\end{proof}

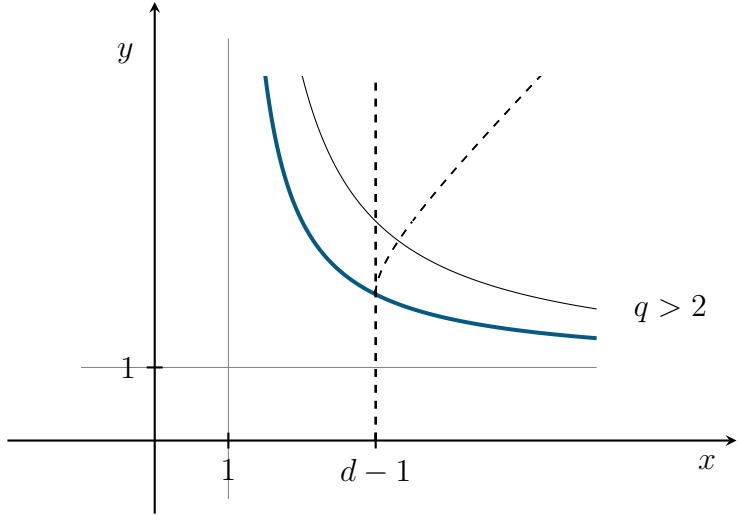
\begin{figure}[h!]
\centering
\begin{tikzpicture}[scale=1.4]

    \begin{axis}[
            axis lines=middle,
            samples=200,
            mystyle2,
            axis equal image,
            axis line style = thick,
            every tick/.style={
                thick,
            },
            axis on top = true
        ]
        % \addplot[white,domain=1.4:6] {1/(x-1) + 1};
        \addplot[name path=f,  darkerblue, domain=1.435:6, line width=1.5 pt] {2/(x-1) + 1};
        % \addplot[white,name path=fbal, line width=0.1 pt, domain=1.435:3] {2/(x-1) + 1};
        % \addplot[white,name path=fjobb, line width=0.1 pt, domain=3:6] {2/(x-1) + 1};
        \path[name path=g] (axis cs:1.5,2/0.5+1) -- (axis cs:6,2/0.5+1);
        % \path[name path=gbal] (axis cs:1.5,2/0.5+1.6) -- (axis cs:3,2/0.5+1.6);
        % \path[name path=gjobb] (axis cs:3,2/0.5+1.6) -- (axis cs:6,2/0.5+1.6);
   
        \addplot [domain=2.01:22,dashed, line width = 0.7 pt] ({x/((x-2)/((x-1)^(1/2)-1)-1)+1},{(x-2)/((x-1)^(1/2)-1)});
        \addplot [domain=2.01:23.5,dashed, line width = 0 pt, name path = fazishatar] ({x/((x-2)/((x-1)^(1/2)-1)-1)+1},{(x-2)/((x-1)^(1/2)-1)});
        \addplot[white,domain=1.8:6] {5/(x-1) + 1};
        % \node (A) at (axis cs:7,1.1) {$q=1$};
        \node (A) at (axis cs:7,1.8) {$q>2$};
        \node (A) at (axis cs:7.5,-0.3) {$x$};
        \node (A) at (axis cs:-0.4,5.3) {$y$};
%         \path[name path=fuggobal] (axis cs:3,0) -- (axis cs:3,2);
%         \path[name path=fuggojobb] (axis cs:6,0) -- (axis cs:6,2/0.55+1);
       
        % asszimptotak;
        \draw[gray, line width=0.2 pt] (axis cs:1,-0.8) -- (axis cs:1,5.5);
        \draw[gray, line width=0.2 pt] (axis cs:-1,1) -- (axis cs:6,1);

        % \addplot[white] fill between[of=fjobb and fazishatar]; % rank2
        % \addplot[white] fill between[of=fazishatar and gjobb]; % rank2
        % \addplot[white] fill between[of=fbal and gbal]; % rank2
        % \draw[draw=none,fill=gray!20] (axis cs:1,0) rectangle (axis cs:6,1); % forest
        % \addplot [domain=3:6,name path = also] {1};
        \addplot [domain=1.8:6,name path = felso] {4/(x-1) + 1};
        \addplot[white, domain=3:6] fill between[of=also and felso]; % harmadik szin
       
        % \draw[black,line width=1pt] (axis cs:6,0) -- (axis cs:1,0) -- (axis cs:1,1) -- (axis cs:6,1);

        \draw[dashed,line width=1pt] (axis cs:3,0) -- (axis cs:3,5);
        \draw[draw=none,fill=white] (axis cs:1.2,5) rectangle (axis cs:6,6); % forest

    \end{axis}

\end{tikzpicture} 
\caption{A generic hyperbola intersecting the line $x=d-1$ and curve of the phase transition.}
\label{fig:without_colors}
\end{figure}
\bigskip

We decompose the proof of Theorem~\ref{th: phase-transition} into three propositions dealing with $w=w_c$, $0\leq w<w_c$ and $w>w_c$.

\begin{Prop} \label{prop:RC_critical}
For $q\geq 2$ and $w=w_c$ we have 
$\Phi_{d,q,w}=q\left(1+\frac{w}{q}\right)^{d/2}$.
\end{Prop}

\begin{proof}
Let us assume that $q>2$, the statement for $q=2$ follows by continuity. We show that for $w=w_c$ the function $\Phi_{d,q,w}(t)$ has a global maximizer in $(0,\pi/2)$ with value $\Phi_{d,q,w}(0)=q\left(1+\frac{w}{q}\right)^{d/2}$. 
We know that for $w=w_c$ there is a $t_1\in \left[0,\frac{\pi}{2}\right]$ such that
$$\frac{a_{q,w,2}(t_1)}{-b_{q,w,1}(t_1)}=\left(\frac{\nu_1}{\nu_2}\right)^{1/d}.$$
This means that
$$\frac{\sqrt{1+\frac{w}{q}}\cos(t_1)-\sqrt{\frac{w}{q(q-1)}}\sin(t_1)}{\sqrt{1+\frac{w}{q}}\sin(t_1)-\sqrt{\frac{(q-1)w}{q}}\cos(t_1)}=(q-1)^{-1/d}\leq 1.$$
Note that if $q=2$, then $t_1=\frac{\pi}{4}$ is a solution. If $q>2$, then
$$\frac{\sin(t_1)}{\cos(t_1)}\geq \frac{\sqrt{1+\frac{w}{q}}+\sqrt{\frac{(q-1)w}{q}}}{\sqrt{1+\frac{w}{q}}+\sqrt{\frac{w}{q(q-1)}}}>1$$
showing that $t_1\in \left(\frac{\pi}{4},\frac{\pi}{2}\right)$. 
Let $t_2=2t_1-\frac{\pi}{2}\in \left(0,\frac{\pi}{2}\right)$. By Lemma~\ref{symmetry-mixed-state} we know that $\Phi_{d,q,w_c}(t)=\Phi_{d,q,w_c}(t_2-t)$. (For an example of the graph of a function $\Phi_{d,q,w_c}(t)$ see Figure~\ref{fig:function_graph}.)

\begin{figure}[h!]
\includegraphics[width=0.5\textwidth]{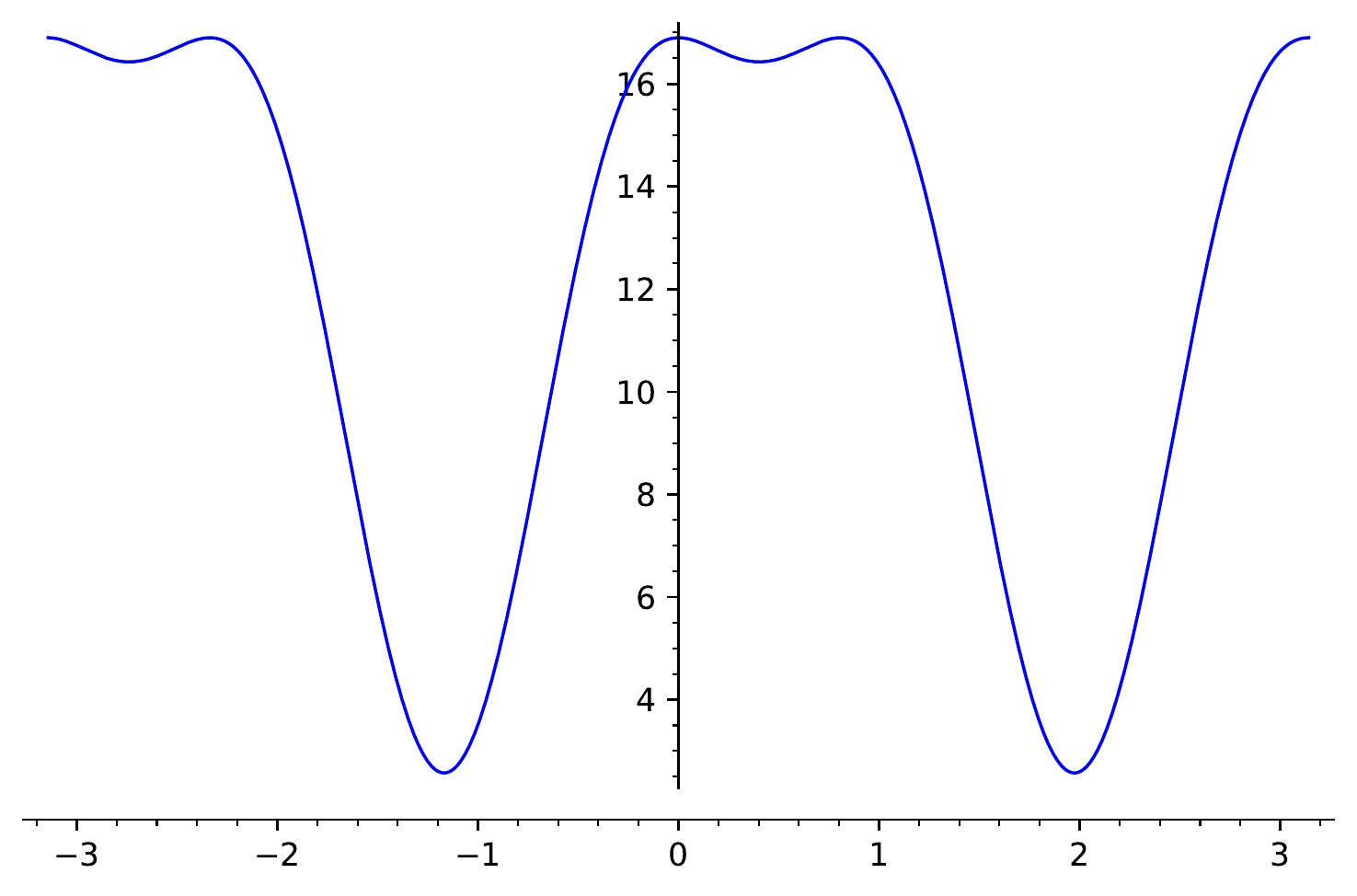}
\caption{For $d=4$ and $q=10$ we have $w_c=3$. The graph of the trigonometric polynomial $\Phi_{4,10,3}(t)$ is depicted in the figure.}
\label{fig:function_graph}
\end{figure}

We know that 
$$\frac{\partial}{\partial t}\Phi_{d,q,w_c}(t)\bigg|_{t=0}=0.$$
This immediately implies that
$$\frac{\partial}{\partial t}\Phi_{d,q,w_c}(t)\bigg|_{t=t_2}=0.$$
The equation $\Phi_{d,q,w_c}(t)=\Phi_{d,q,w_c}(t_2-t)$ also implies that 
$$\frac{\partial}{\partial t}\Phi_{d,q,w_c}(t)\bigg|_{t=t_2/2}=0.$$
This means that  a computation similar to the one in Section~\ref{sec: trig-Bethe} gives  that if
$$R(t)=-(q-1)\frac{b_{q,w,2}(t)}{b_{q,w,1}(t)}\ \ \ \text{and}\ \ \ S(t)=\frac{a_{q,w,1}(t)}{a_{q,w,2}(t)}=\frac{(1+w)R(t)+q-1}{R(t)+q+w-1},$$
then the values $R=R(0),R\left(\frac{t_2}{2}\right),R(t_2)$ are all solutions of  the equation
$$R=\left(\frac{(1+w)R+q-1}{R+w+q-1}\right)^{d-1}.$$
The values  $R(0),R\left(\frac{t_2}{2}\right),R(t_2)$ are at least $1$, because  by Lemma~\ref{Lemma:bethesigns} they are non-negative, and $S(t)>1$ whenever $t\in \left[0,\frac{\pi}{2}\right]$ and both $a_{q,w,1}(t),a_{q,w,2}(t)>0$. The equation 
$$R=\left(\frac{(1+w)R+q-1}{R+w+q-1}\right)^{d-1}$$
has at most $3$ solutions satisfying $R\geq 1$ by Lemma~\ref{number-solutions} which means that there is no other  $t'\in \left(0,\frac{\pi}{2}\right)$ that is a local maximizer or minimizer of $\Phi_{d,q,w}(t)$. Note that $\frac{d^2}{dt^2}\Phi_{d,q,w_c}\bigg|_{t=0}<0$ since $w_c(q)<\frac{q}{d-2}$ by Lemma~\ref{qw-xy-reparametrization}. So at $\frac{t_2}{2}$ we have a local minimum, and at $t_2$ we have a local maximum. Hence $$\Phi_{d,q,w_c}=\Phi_{d,q,w_c}(0)=\Phi_{d,q,w_c}(t_2)=q\left(1+\frac{w}{q}\right)^{d/2}.$$
\end{proof}

\begin{Prop}
For $q\geq 2$ and $0\leq w\leq w_c$ we have
$\Phi_{d,q,w}=q\left(1+\frac{w}{q}\right)^{d/2}$.
\end{Prop}

\begin{proof}
We will describe the pairs $(q,w)$ for which  $\Phi_{d,q,w}=q\left(1+\frac{w}{q}\right)^{d/2}$.
To do this it is better to  use the Tutte polynomial $T_G(x,y)$ instead of $Z_G(q,w)$ with $q=(x-1)(y-1)$ and $w=y-1$. Recall that the connection between the Tutte polynomial and the partition function of the random cluster model is the following:
$$T_G(x,y)=(x-1)^{-k(E)}(y-1)^{-v(G)}Z_G((x-1)(y-1),y-1).$$
Then for $q\geq 2$ and an essentially large girth sequence of $d$-regular graphs $(G_n)_n$ the statement
$$\lim_{n\to \infty} Z^{(2)}_{G_n}(q,w)^{1/v(G_n)}=\lim_{n\to \infty} Z_{G_n}(q,w)^{1/v(G_n)}=q\left(1+\frac{w}{q}\right)^{d/2}$$
is equivalent with
$$\lim_{n\to \infty}T_{G_n}(x,y)^{1/v(G_n)}=x\left(1+\frac{1}{x-1}\right)^{d/2-1}.$$
This is independent of $y$. The Tutte polynomial has only non-negative coefficients \cite{tutte1954contribution}, so if this limit value holds true for $(x,y_1)$ and $(x,y_2)$, then so for every $y\in [y_1,y_2]$. Note that for $x\geq d-1$ and $y=1$ this was indeed proved by Bencs and Csikv\'ari \cite{bencs2021evaluations}. In fact, we do not even need to use this result since for $q=1$ this statement is trivial. 
By Lemma~\ref{qw-xy-reparametrization} the curve $(q,w_c(q))$ for $q\geq 2$ reparametrized with $x$ and $y$ is the graph of a monotone increasing function on the interval $[d-1,\infty)$, see the dashed line on Figure 1.
In particular, for  $q\geq 2$ the part of the hyperbola $(x-1)(y-1)=q$ with $0\leq w=y-1\leq w_c$ goes under this curve implying $\Phi_{d,q,w}=q\left(1+\frac{w}{q}\right)^{d/2}$.
\end{proof}

\begin{Rem} \label{1q2-small-w}
We remark that the same argument also gives that if $1<q<2$ and \\ $0\leq w\leq \frac{q}{d-2}$, then for an essentially large girth sequence of $d$-regular graphs $(G_n)_n$ we have
$$\lim_{n\to \infty}\frac{1}{v(G_n)}\ln Z^{(2)}_{G_n}(q,w)=\lim_{n\to \infty}\frac{1}{v(G_n)}\ln Z_{G_n}(q,w)=q\left(1+\frac{w}{q}\right)^{d/2}.$$

\end{Rem}

\begin{Prop}
For $q\geq 2$ and $w>w_c$ we have
$\Phi_{d,q,w}>q\left(1+\frac{w}{q}\right)^{d/2}$. Furthermore, the function $\frac{\partial}{\partial w}\Phi_{d,q,w}$ has a discontuinity at $w=w_c$ if $q>2$.
\end{Prop}

\begin{proof}
Consider the function
$$h(w,t)=\left(1+\frac{w}{q}\right)^{-d/2}\Phi_{d,q,w}(t).$$
We show that it is a strictly monotone increasing function in $w$ for every $t\in \left(0,\frac{\pi}{2}\right)$. 
By definition
\[
    h(w,t)=\left(\cos(t)+\sqrt{\frac{(q-1)w}{q+w}}\sin(t)\right)^d+(q-1)\left(\cos(t)-\sqrt{\frac{w}{(q-1)(q+w)}}\sin(t)\right)^d.
\]
Then $\frac{\partial h}{\partial w}$ is given as
$$\frac{dq\sqrt{q-1}\sin(t)}{\sqrt{w(q+w)^3}}\left(\left(\cos(t)+\sqrt{\frac{(q-1)w}{q+w}}\sin(t)\right)^{d-1}-\left(\cos(t)-\sqrt{\frac{w}{(q-1)(q+w)}}\sin(t)\right)^{d-1}\right).$$
This is positive  if $t\in \left(0,\frac{\pi}{2}\right)$ since
$$\cos(t)+\sqrt{\frac{(q-1)w}{q+w}}\sin(t)> \left|\cos(t)-\sqrt{\frac{w}{(q-1)(q+w)}}\sin(t)\right|.$$
Note that for $q>2$ there is a $t_0(w_c)\in \left(0,\frac{\pi}{2}\right)$ such that $\Phi_{d,q,w_c}(t_0(w_c))=\Phi_{d,q,w_c}(0)$, that is, 
$h(w_c,t_0(w_c))=q$. Then for $w>w_c$ we have $h(w,t_0(w_c))>q$ which gives that
$$\Phi_{d,q,w}\geq \Phi_{d,q,w}(t_0(w_c))>q\left(1+\frac{w}{q}\right)^{d/2}.$$
For $q=2$ we know that $w_c=\frac{2}{d-2}$ and for $w> \frac{q}{d-2}=\frac{2}{d-2}$ we have $\frac{\partial}{\partial t}\Phi_{d,q,w}(t)\bigg|_{t=0}=0$ and  $\frac{\partial}{\partial t^2}\Phi_{d,q,w}(t)\bigg|_{t=0}<0$, so at $t=0$ we have a local minimum, thus $\Phi_{d,q,w}>\Phi_{d,q,w}(0)$ for $w>w_c$.

Next we prove the claim about $\frac{\partial}{\partial w}\Phi_{d,q,w}$. Let $w>w_c$ such that $w-w_c$ is small enough, namely it satisfies
$$h(w,t_0(w_c))\geq h(w_c,t_0(w_c))+\frac{1}{2}(w-w_c)\frac{\partial}{\partial w}h(w,t_0(w_c))\bigg|_{w=w_c}=q+\frac{1}{2}(w-w_c)\frac{\partial}{\partial w}h(w,t_0(w_c))\bigg|_{w=w_c}.$$
Then
\begin{align*}
\frac{\Phi_{d,q,w}-\Phi_{d,q,w_c}}{w-w_c}&\geq \frac{1}{w-w_c}\left(\left(1+\frac{w}{q}\right)^{d/2}h(w,t_0(w_c))-q\left(1+\frac{w_c}{q}\right)^{d/2}\right)\\
&\geq \frac{1}{w-w_c}\left(\left(1+\frac{w}{q}\right)^{d/2}\left(q+\frac{1}{2}(w-w_c)\frac{\partial}{\partial w}h(w,t_0(w_c))\bigg|_{w=w_c}\right)-q\left(1+\frac{w_c}{q}\right)^{d/2}\right)
\end{align*}
From this it follows that
$$\frac{\partial}{\partial w^+}\Phi_{d,q,w}\bigg|_{w=w_c}\geq \frac{\partial}{\partial w^-}\Phi_{d,q,w}\bigg|_{w=w_c}+\frac{1}{2}\left(1+\frac{w_c}{q}\right)^{d/2}\frac{\partial}{\partial w}h(w,t_0(w_c))\bigg|_{w=w_c}.$$
\end{proof}

\begin{Rem} 
It is well-known that if $q=2$, then there is a second order phase transition, that is, $\frac{\partial}{\partial w}\Phi_{d,2,w}$ is continuous, but $\frac{\partial^2}{\partial w^2}\Phi_{d,2,w}$ is discontinuous at $w=w_c=\frac{2}{d-2}$. For details see Chapter 4.8 of \cite{baxter2016exactly}. 
\end{Rem}

\subsection{Examples} In this section we give some examples for the theorems we proved.

\begin{Ex}
Let $d=8$, $q=5$ and $w=1$. Then the vector 
$$\vecv(0)=(10.368, 0, 1.728, 1.058, 0.936, 0.749, 0.615, 0.501, 0.409),$$
where we kept only the first three digits everywhere. 
Note that $10.368=5\cdot \left(1+\frac{1}{5}\right)^{8/2}.$
So for every $8$-regular graph $G$ we have
$$Z^{(2)}_G(5,1)=F_G(10.368, 0, 1.728, 1.058, 0.936, 0.749, 0.615, 0.501, 0.409).$$
Using $t_0=0.6619549492373429$ we get the vector
$$\vecv(t_0)=(16.277, 0, 0.433, -0.496, 0.581, -0.679, 0.794, -0.929, 1.086)$$
again only keeping the first $3$ digits everywhere. A more precise value of the first coordinate is $16.277748757985485$, and so this is $\Phi_{8,5,1}$. Note that the sign structure of $\vecv(t_0)$ shows that
$$Z^{(2)}_G(5,1)=F_G(16.277, 0, 0.433, -0.496, 0.581, -0.679, 0.794, -0.929, 1.086)\geq 16.277^{v(G)}$$
for every $8$-regular graph $G$.
\end{Ex}

\begin{Ex} Let $d=4$, $q=5$ and $w=3$. Then
$$\vecv(0)=(12.8, 0, 4.8, 4.409, 5.85)$$
where we again kept only the first three digits everywhere. This time $t_0=0.8316331320342567$ and $\Phi_{4,5,3}=16.315621073058985$ while
$$\vecv(t_0)=(16.315, 0, 1.878, -3.867, 8.176).$$
In this case $t_1=1.06627054934707$ and the corresponding vector
$$\vecv(t_1)=(15.010, -2.835, 0.994, -2.454, 11.249).$$
For the complete graph $K_5$ on $5$ vertices the subgraph counting polynomial looks like as follows:
\begin{align*}
&F_{K_5}(15.010, -2.835, 0.994, -2.454, 11.249\ |\ z)=180176.234z^{20} + 85764.618z^{18} \\
&+28876.392z^{16} + 15784.587z^{14} + 10454.536z^{12} + 9093.510z^{10} + 13949.743z^{8} + 28103.222z^6 \\
&+ 68600.498z^4 + 271865.398z^2 + 762087.303
\end{align*}
All zeros of this polynomial have absolute value approximately $1.0747696$. Of course, we could have used any $4$-regular graph instead of $K_5$ (see Figure~\ref{fig:zeros}).

\begin{figure}
\includegraphics[width=0.5\textwidth]{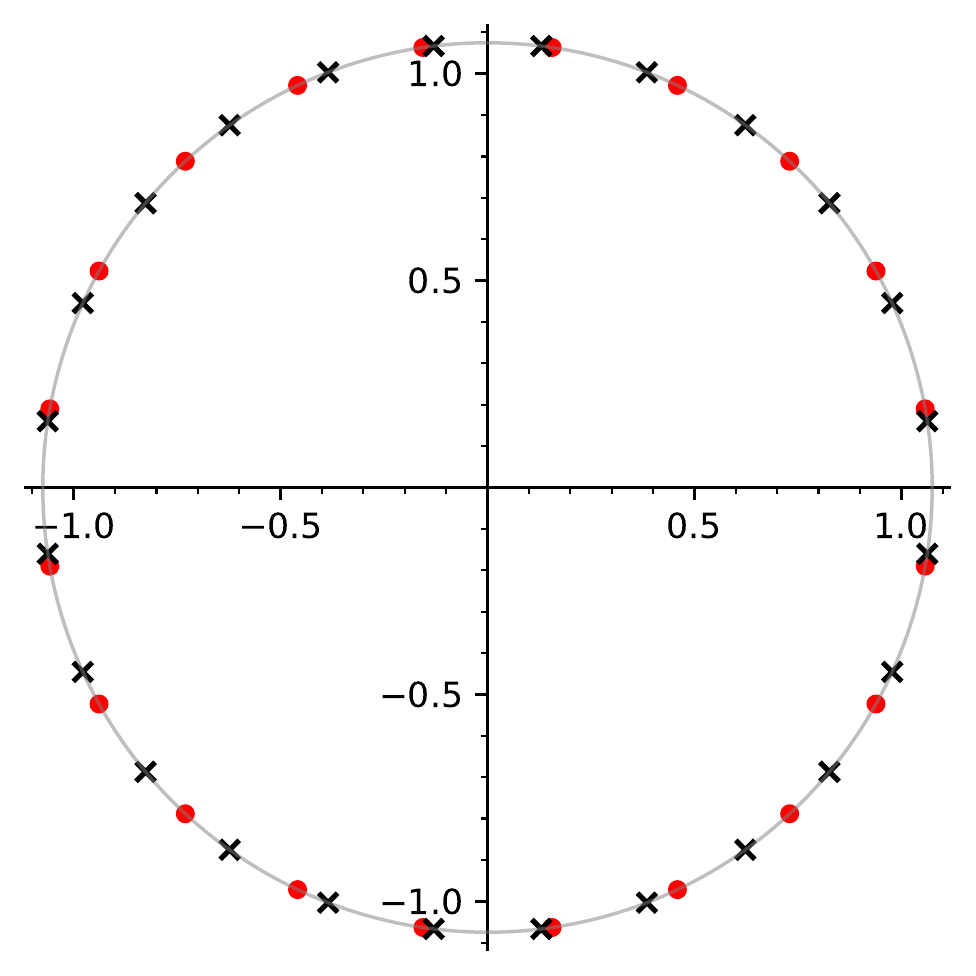}
\caption{The zeros of $F_{G}(15.010, -2.835, 0.994, -2.454, 11.249\ |\ z)$, where $G$ is $K_5$ (red) and $G$ is the octahedron (black x).}
\label{fig:zeros}
\end{figure}
\end{Ex}

\begin{Ex}
Let $d=4$ and $q=5$ again, but let $w=w_c=2$. Then
$$\vecv(0)=(9.8, 0, 2.8, \sqrt{5.04}, 2.6),$$
$t_0=0.5575988373258864$ and
$$\vecv(t_0)=(9.8, 0, 2.8, -\sqrt{5.04}, 2.6).$$
We have $t_1=1.06419757674722$ and
$$\vecv(t_1)=(8, -\sqrt{4.5}, 1, -\sqrt{4.5}, 8).$$
One can check that
\begin{align*}
&F_{K_5}(8, -\sqrt{4.5}, 1, -\sqrt{4.5}, 8\ |\ z)=
32768z^{20} + 23040z^{18} +
11070z^{16} + 6647.5z^{14} +\\
&4620z^{12} + 3927z^{10} + 4620z^8 + 6647.5z^6 + 11070z^4 + 23040z^2 + 32768
\end{align*}
and all of its zeros have absolute value $1$.

\end{Ex}

\section{Selected remarks about the interval $1<q<2$}\label{sec:related}

In this section we collected several remarks about the interval $1<q<2$.

\subsection{Two different quantities} In this section we aim to explain a seemingly negligible thing that makes the interval $q\geq 2$ and $1<q<2$ really different.

Once again let $N=M_2'$ and $\vecmu=\underline{\nu}_2$, and parametrize the distribution $h$ in the Bethe recursion as follows:
$$h=\left(\frac{R}{R+q-1},\frac{q-1}{R+q-1}\right).$$
Then
$$\mathrm{BP}(h)_{1}=\frac{1}{z_h}\left(\frac{(1+w)R+q-1}{R+q-1}\right)^{d-1}$$
and
$$\mathrm{BP}(h)_{2}:=\frac{1}{z_h}(q-1)\left(\frac{R+w+q-1}{R+q-1}\right)^{d-1}.$$
If $\mathrm{BP}(h)=h$, then by dividing the Bethe recursions for $h_1$ and $h_2$ we get that
$$R=\left(\frac{(1+w)R+q-1}{R+w+q-1}\right)^{d-1}.$$
We remark that if we study the Potts model $N=M=wI_q+J_q$ and $\mu\equiv 1$ with
$$h=\left(\frac{R}{R+q-1},\frac{1}{R+q-1},\dots ,\frac{1}{R+q-1}\right),$$
then we would have arrived to the same equation.
Let $\mathcal{R}_{d,q,w}$ be the set of non-negative solutions of this equation. Let $\mathcal{R}^*_{d,q,w}$ be the solutions satisfying also that $R\geq 1$. 
Let us introduce the notation
\begin{align*}
\overline{\mathbb{F}}(R,d,q,w)&=\left(\frac{(1+w)R+q-1}{\sqrt{(1+w)(R^2+q-1)+2R(q-1)+(q-1)(q-2)}}\right)^d+\\
&\ \ \ +(q-1)\left(\frac{R+q+w-1}{\sqrt{(1+w)(R^2+q-1)+2R(q-1)+(q-1)(q-2)}}\right)^d
\end{align*}
Then we know that
$$\Phi_d(M'_2,\nu_2)=\max_{R\in \mathcal{R}_{d,q,w}}\overline{\mathbb{F}}(R,d,q,w).$$
For later use let us also introduce 
$$\Phi^*_d(M'_2,\nu_2)=\max_{R\in \mathcal{R}^*_{d,q,w}}\overline{\mathbb{F}}(R,d,q,w).$$
Similarly, we can consider the pair
$$\Phi_{d,q,w}=\max_{t\in [0,2\pi]}\Phi_{d,q,w}(t)\ \ \text{and}\ \ \ \Phi^*_{d,q,w}=\max_{t\in \left[0,\frac{\pi}{2}\right]}\Phi_{d,q,w}(t).$$
In case of $q\geq 2$ we have
$$\Phi_d(M'_2,\nu_2)=\Phi_{d,q,w}=\Phi^*_d(M'_2,\nu_2)=\Phi^*_{d,q,w}.$$
But when $1<q<2$ we have
$$\Phi_d(M'_2,\nu_2)=\Phi_{d,q,w}>\Phi^*_d(M'_2,\nu_2)=\Phi^*_{d,q,w}.$$
While it is still true that for an essentially large girth sequence of $d$-regular graphs we have 
$$\lim_{n\to \infty}\frac{1}{v(G_n)}\ln Z^{(2)}_{G_n}(q,w)=\Phi_{d,q,w},$$
we actually believe that 
$$\lim_{n\to \infty}\frac{1}{v(G_n)}\ln Z_{G_n}(q,w)=\Phi^*_{d,q,w}.$$
This means that the rank $2$ approximation is not good enough in the interval $1<q<2$.

Nevertheless, by Remark~\ref{1q2-small-w} we know that for $1<q<2$ and $0\leq w\leq \frac{q}{d-2}$ we have
$$\lim_{n\to \infty}\frac{1}{v(G_n)}\ln Z_{G_n}(q,w)=q\left(1+\frac{w}{q}\right)^{d/2}.$$
We remark that this result is compatible with the conjecture 
$$\lim_{n\to \infty}\frac{1}{v(G_n)}\ln Z_G(q,w)=\Phi^*_{d,q,w}$$
since for the function $\Phi_{d,q,w}(t)$ we have
$$\frac{\partial}{ \partial t}\Phi_{d,q,w}(t)\bigg|_{t=0}=0\ \ \ \text{and}\ \ \ \frac{\partial^2}{\partial t^2}\Phi_{d,q,w}(t)\bigg|_{t=0}=d\left(1+\frac{w}{q}\right)^{d/2-1}((d-2)w-q)$$
which is negative if $w<\frac{q}{d-2}$ and positive $w>\frac{q}{d-2}$. So in the first case we get that $t=0$ is a local maximum, in the second case it is a local minimum. 
\bigskip

\noindent \textbf{Acknowledgment.} We are very grateful to the anonymous reviewers for the careful reading and the suggestions leading to a significant improvement in the presentation of this paper.

\bibliography{hivatkozat}
\bibliographystyle{plain}

\end{document}